\documentclass[11 pt, oneside]{amsart}
\usepackage[utf8]{inputenc}
\usepackage[margin=3cm]{geometry}
\makeatletter
\def\l@subsection{\@tocline{2}{0pt}{2.5pc}{5pc}{}}
\def\l@subsubsection{\@tocline{2}{0pt}{5pc}{7.5pc}{}}
\makeatother

\usepackage{stmaryrd}
\usepackage{amsmath,amssymb,amsthm, mathrsfs, mathtools, comment}
\usepackage{amscd,mathtools}
\usepackage{geometry}
\usepackage{graphicx,epstopdf}
\usepackage{color,xcolor}
\usepackage{enumerate}
\usepackage{xspace}
\usepackage{tipa}
\usepackage[user,titleref]{zref}
\usepackage{epsfig}

  \usepackage{tikz}
  \usetikzlibrary{calc,through,backgrounds, arrows}
  \usetikzlibrary{decorations.pathmorphing, decorations.pathreplacing,positioning,fit}

\usepackage{hyperref}
\usepackage{float}

\newcommand{\calA}{\mathcal{A}}

\newcommand{\calC}{\mathcal{C}}

\newcommand{\calG}{\mathcal{G}}

\newcommand{\calM}{\mathcal{M}}

\newcommand{\calP}{\mathcal{P}}

\newcommand{\calR}{\mathcal{R}}
\newcommand{\calS}{\mathcal{S}}

\newcommand{\calV}{\mathcal{V}}

\newcommand{\ZZ}{\mathbb{Z}}

\newtheorem{theorem}{Theorem}[section]
\newtheorem{proposition}[theorem]{Proposition}
\newtheorem{corollary}[theorem]{Corollary}
\newtheorem{lemma}[theorem]{Lemma}

\newtheorem{question}[theorem]{Question}

\newtheorem{introthm}{Theorem}
\newtheorem{introcor}[introthm]{Corollary}

\newtheorem{introconj}[introthm]{Conjecture}

\theoremstyle{definition}
\newtheorem{definition}[theorem]{Definition}

\newtheorem*{claim*}{Claim}

\newtheorem{example}[theorem]{Example}

\newtheorem*{question*}{Question}
\newtheorem*{answer*}{Answer}
\newtheorem*{application*}{Application}

\theoremstyle{remark}
\newtheorem{remark}[theorem]{Remark}
\newtheorem*{remark*}{Remark}

\newcommand{\figref}[1]{Figure~\ref{#1}}

\DeclareMathOperator{\genus}{genus}
\DeclareMathOperator{\End}{End}

\DeclareMathOperator{\tw}{tw}

\DeclareMathOperator{\Map}{Map}

\DeclareMathOperator{\Div}{Div}
\DeclareMathOperator{\id}{id}

\newcommand{\Aut}{\ensuremath{\operatorname{Aut}}\xspace}

\DeclareMathOperator{\Homeo}{Homeo}

\DeclareMathOperator{\diam}{diam}

\DeclareMathOperator{\MapSig}{\Map(\Sigma)}

\newcommand{\param}{{\mathchoice{\mkern1mu\mbox{\raise2.2pt\hbox{$
\centerdot$}}
\mkern1mu}{\mkern1mu\mbox{\raise2.2pt\hbox{$\centerdot$}}\mkern1mu}{
\mkern1.5mu\centerdot\mkern1.5mu}{\mkern1.5mu\centerdot\mkern1.5mu}}}

\DeclarePairedDelimiterX{\norm}[1]{\lvert}{\rvert}{#1}
\DeclarePairedDelimiterX{\Norm}[1]{\lVert}{\rVert}{#1}

\newcommand{\Cnp}{{C_{\rm np}}}
\newcommand{\CnpS}{{C_{\rm np}(\Sigma)}}

\newcommand{\CAT}{\ensuremath{\operatorname{CAT}(0)}\xspace} \newcommand{\mute}[1]{}

\begin{document}

\title[The non-peripheral curve graph and divergence in big mapping class groups]
{The non-peripheral curve graph and divergence in big mapping class groups}

%
\author{Assaf Bar-Natan}
\email{assaf.barnatan@gmail.com}

\author{Yulan Qing}
\address{Department of Mathematics, University of Tennessee}
\email{yqing@utk.edu}

\author{Kasra Rafi}
\address{Department of Mathematics, University of Toronto}
\email{rafi@math.toronto.edu}
%



\begin{abstract}
We introduce a
numerical invariant $\zeta(\Sigma)$ measuring the end-complexity of $\Sigma$ 
and use it to organize coarse-geometric features of $\Map(\Sigma)$.
Our main tool is the \emph{non-peripheral curve graph} $\Cnp(\Sigma)$, whose vertices are those essential simple closed
curves that cannot be pushed out of every compact subsurface, with edges given by disjointness.
Assuming $\Map(\Sigma)$ is CB-generated and $\zeta(\Sigma)\ge 5$, we prove that $\Cnp(\Sigma)$ is connected, has infinite
diameter, is Gromov hyperbolic, and that the $\Map(\Sigma)$-action has unbounded orbits.  As applications, we show that
if $\zeta(\Sigma)\ge 4$ then $\Map(\Sigma)$ has infinite coarse rank, and if $\zeta(\Sigma)\ge 5$ then $\Map(\Sigma)$ has
at most quadratic divergence, hence is one-ended.  
\end{abstract}

\maketitle

\tableofcontents

\section{Introduction}

Let $\Sigma$ be a connected, orientable, infinite-type surface such that the end space of
$\Sigma$ is stable (Definition~\ref{def:stable}). Let $\Map(\Sigma)$ denote its mapping class group:
the group of orientation-preserving homeomorphisms up to isotopy. Equipped with
the compact-open topology, $\Map(\Sigma)$ is a Polish group \cite{Rosendal}. In contrast to the
finite-type case, $\Map(\Sigma)$ is never countably generated. Nevertheless, for many infinite-type
surfaces the associdated $\MapSig$ admits a robust large-scale geometry in the sense of Rosendal: if $\Map(\Sigma)$ is
\emph{CB-generated} (i.e.\ generated by a coarsely bounded neighborhood of the identity together
with finitely many additional elements), then any two word metrics coming from CB-generating sets
are quasi-isometric, so the quasi-isometry type of $\Map(\Sigma)$ is well-defined
\cite{Rosendal,MR} within this class of generating sets.

\subsection*{CB-generation and coarse geometry}
In \cite{MR} Mann and Rafi give a classification of those infinite-type surfaces
$\Sigma$ for which $\Map(\Sigma)$ is coarsely boundedly generated, or CB-generated. In this setting,
one can equip $\Map(\Sigma)$ with a word metric coming from any coarsely bounded generating set, and
the resulting quasi-isometry type is independent of the choice of CB generators \cite{Rosendal,MR}.
This classification opens the door to a systematic study of quasi-isometry invariants for big mapping
class groups, in direct analogy with the coarse geometry of finitely generated groups.

Several recent works develop this viewpoint. Grant--Rafi--Verberne \cite{GRV} investigate quasi-isometry
invariants such as asymptotic dimension and introduce \emph{essential shifts} (shift maps that are
countable and cannot be pushed to infinity); existence of essential shifts implies the existence of
high-dimensional quasi-flats in $\Map(\Sigma)$. Related results on asymptotic dimension for various 
classes of big mapping class groups appear in the recent preprint \cite{KS25}, where Kopreski--Shaji construct arc-and-curve models for locally bounded Polish subgroups of $\Map(\Sigma)$, and use 
these models to show (in particular) that the asymptotic dimension of a CB-generated $\Map(\Sigma)$ 
is infinite unless $\Map(\Sigma)$ is CB. 
We emphasize that asymptotic dimension alone does not detect coarse rank: in general, infinite 
asymptotic dimension does not rule out the possibility that the coarse rank is finite. 
In a different direction, Horbez--Qing--Rafi \cite{HQR}
classify when (CB-generated) big mapping class groups admit nonelementary actions on hyperbolic spaces
and derive applications. It turns out that $\Map(\Sigma)$ has a non-trivial action on a Gromov hyperbolic
space if and only if $\Sigma$ contains non-displaceable subsurfaces. These results relate the topology
of $\Sigma$ to the large-scale geometry of $\Map(\Sigma)$.

\subsection*{A complexity for end spaces}
In this paper we introduce a numerical invariant $\zeta(\Sigma)$ (
Section~\ref{sec:nondisplaceable},  Definition~\ref{countcomplexity}), which measures the
``end-complexity'' of $\Sigma$ in a way that is tailored to the CB setting. Concretely, $\zeta(\Sigma)$
is defined using the minimal anchor surface $K_0$ from \cite{MR} such that the stabilizer
of $K_0$ is a CB subgroup of $\Map(\Sigma)$. One should think of $\zeta(\Sigma)$ as an analogue of the
finite-type complexity parameter ``number of boundary components'' (or punctures), which strongly
influences the coarse geometry of $\Map(S)$ for finite-type surfaces. We show that $\zeta(\Sigma)$
predicts several coarse-geometric features of $\Map(\Sigma)$.

\subsection*{Main results}
The first application concerns coarse rank. For finitely generated mapping class groups, the (coarse)
rank is the largest $n$ such that $\ZZ^n$ quasi-isometrically embeds into the group, and it is finite for
finite-type mapping class groups \cite{BM}. In the present setting, CB-generation gives a well-defined
quasi-isometry type of word metrics, and hence coarse rank is again a quasi-isometry invariant.

\begin{introthm}[Mapping class groups with infinite coarse rank]\label{introthm:Coarse}
Suppose $\Sigma$ is stable, $\Map(\Sigma)$ is CB-generated, and $\zeta(\Sigma)\ge 4$.
Then $\Map(\Sigma)$ has infinite coarse rank.
\end{introthm}

Another source of infinite coarse rank is the existence of an essential shift as discussed in \cite{GRV}.
In fact, for $\zeta(\Sigma)=1,2$, the mapping class group is either coarsely bounded or contains an essential shift. For $\zeta(\Sigma)=3$, there are examples of surfaces 
where $\Map(\Sigma)$ is unbounded and Gromov hyperbolic which implies that the coarse rank is 1 (see \cite{Sch24}). 
This yields the following possible dichotomy in the CB-generated setting where the case of $\zeta(\Sigma) =3$ need a closer examination.

\begin{introconj}\label{introcor:hyperbolic-or-rank}
Suppose $\Sigma$ is stable and $\Map(\Sigma)$ is CB-generated. Then $\Map(\Sigma)$ is either Gromov
hyperbolic or it has infinite coarse rank.
\end{introconj}

Our second application concerns divergence, a quasi-isometry invariant that measures how long detours
must be in order to connect two points while avoiding a large ball. Gersten introduced divergence as a
quasi-isometry invariant in \cite{Ger94}; in Euclidean space divergence is linear, while in
$\delta$-hyperbolic spaces divergence is at least exponential. Over the past two decades it has become
clear that quadratic divergence is remarkably common: Gersten constructed CAT(0) examples with quadratic
divergence in \cite{Ger94}, and in \cite{Ger94B} he showed that fundamental groups of closed geometric
$3$-manifolds have divergence that is either linear, quadratic, or exponential, with quadratic divergence
corresponding to graph manifold groups; this characterization was strengthened by Kapovich--Kleiner--Leeb
\cite{KKL}. In higher rank, Druţu--Mozes--Sapir \cite{DMS} develop a robust definition of divergence for
geodesic metric spaces and conjecture that many higher rank lattices have linear divergence; and
Behrstock--Charney \cite{BC} show that most right-angled Artin groups have either linear or quadratic
divergence, with the linear case characterized by the defining graph being a join. For mapping class groups and Teichm\"uller space of finite-type surfaces, quadratic divergence was
proved by Behrstock \cite{Beh06} and by Duchin--Rafi \cite{DR}, who give a flexible detour
construction; we will follow the approach of \cite{DR}
(see also \cite{BDM}).

\begin{introthm}[Quadratic divergence bound]\label{introthmDiv}
Suppose $\Sigma$ is stable, $\Map(\Sigma)$ is CB-generated, and $\zeta(\Sigma)\ge 5$.
Then $\Map(\Sigma)$ has at most quadratic divergence.
\end{introthm}

In another word, $\Map(\Sigma)$ is thick in the sense of Behrstock-Druţu-Mosher \cite{BDM}.
As a geometric consequence, we obtain one-endedness (in the sense of ends of the Cayley graph associated
to any CB word metric). As a geometric consequence, we obtain one-endedness
(in the sense of ends of the Cayley graph associated to any CB word metric).

\begin{introcor}\label{introcor:one-ended}
Suppose $\Sigma$ is stable, $\Map(\Sigma)$ is CB-generated, and $\zeta(\Sigma)\ge 5$.
Then $\Map(\Sigma)$ is one-ended.
\end{introcor}

See also \cite{OQW}, where one-endedness is established for avenue surfaces.
These surfaces have $\zeta(\Sigma)=2$, and the associated mapping class groups all contain an
essential shift.

\subsection*{A $\zeta$-based picture.}
A recurring theme in this paper is that $\zeta(\Sigma)$ organizes the coarse geometry 
of $\Map(\Sigma)$ much like finite-type complexity organizes the geometry of $\Map(S)$. 
The following table summarizes the picture for stable surfaces when genus is either zero
or infinite. Recall that, if there is an essential shift, then $\Map(\Sigma)$ has
infinite coarse rank \cite{GRV}.

\medskip
\begin{center}\label{table:intro}
\begin{tabular}{|c|p{0.72\linewidth}|}
\hline
Complexity $\zeta(\Sigma)$ & Group properties of $\Map(\Sigma)$ \\ \hline
$\zeta(\Sigma)=1$ & $\Map(\Sigma)$ is always coarsely bounded \cite{MR}. \\ \hline
$\zeta(\Sigma)=2$ & $\Map(\Sigma)$ is coarsely bounded iff there are no essential shifts. \\ \hline
$\zeta(\Sigma)=3$ &
$\Map(\Sigma)$ is not coarsely bounded and sometimes Gromov hyperbolic. \\ \hline
$\zeta(\Sigma)=4$ & $\Map(\Sigma)$ has infinite coarse rank (Theorem~\ref{introthm:Coarse}). \\ \hline
$\zeta(\Sigma)\ge 5$ & $\Map(\Sigma)$ is one-ended and has at most quadratic divergence
(Theorem~\ref{introthmDiv} and Corollary~\ref{introcor:one-ended}). \\ \hline
\end{tabular}
\end{center}

\subsection*{The non-peripheral curve graph.}
The main tool in this paper is a new curve graph adapted to CB coarse geometry. In the finite-type
setting, Masur--Minsky's curve graph is central: the action of $\Map(S)$ on the curve graph captures
much of the large-scale geometry of $\Map(S)$ and underlies powerful hierarchical structure. For
infinite-type surfaces, finding useful analogues has been a major theme; recent years have produced
several hyperbolic graphs with big mapping class group actions, including \cite{DFV,BNV,HQR} among
others.

A naive adaptation of the classical curve graph does not directly serve coarse geometry in the big
setting. Indeed, the ``usual'' curve graph on $\Sigma$ (vertices all essential curves, edges
disjointness) is algebraically interesting (see, for instance, \cite{DFV,MR} and the discussion
therein), but in the infinite-type cases it has finite diameter (in fact diameter $2$), so the action
of $\Map(\Sigma)$ on this graph does not detect large-scale features of $\Map(\Sigma)$.

Our starting point is that, in the CB framework, the stabilizer of a sufficiently large compact
subsurface is a coarsely bounded subset of $\Map(\Sigma)$, i.e.\ it is \emph{small} from the viewpoint
of coarse geometry. Therefore, to build a curve-type graph that sees unbounded geometry, one should
focus on curves that cannot be pushed out of every compact region. In the finite-type case, the word
\emph{peripheral} means the curve bounds a once-punctured disk; equivalently, it can be pushed into a
cusp away from all compact sets. For infinite type we adopt the latter point of view: a curve (or
compact subsurface) is \emph{peripheral} if for every compact subsurface $K\subset\Sigma$ there exists
$g\in\Map(\Sigma)$ such that $g(\alpha)\cap K=\emptyset$. Curves that fail this are
\emph{non-peripheral}. This definition is tightly related to the notion of \emph{essential shift} from
\cite{GRV}: a shift map is geometrically essential only when its support cannot be pushed into an end
neighborhood, i.e.\ when it is supported on a non-peripheral region. This philosophy also explains why
non-peripheral subsurfaces and curves are the natural objects to examine in CB coarse geometry.

We define the \emph{non-peripheral curve graph} $\Cnp(\Sigma)$ to be the graph whose vertices are
non-peripheral simple closed curves and whose edges connect disjoint pairs. Our first main structural
result is that, at sufficiently high end-complexity, $\Cnp(\Sigma)$ becomes a useful Masur--Minsky type
object.

\begin{introthm}[Geometry of $\Cnp(\Sigma)$]\label{introthm:Cnp}
Let $\Sigma$ be stable, $\Map(\Sigma)$ be CB-generated, and suppose $\zeta(\Sigma)\ge 5$.
Then $\Cnp(\Sigma)$ is connected and has infinite diameter, $\Map(\Sigma)$ acts on $\Cnp(\Sigma)$ with
unbounded orbits, and $\Cnp(\Sigma)$ is Gromov hyperbolic.
\end{introthm}

A related study of Qing--Thomas \cite{QT25} (using and adapting the unicorn-path technology of
Hensel--Przytycki--Webb \cite{HPW}) establishes connectedness and hyperbolicity of $\Cnp(\Sigma)$
for some lower-complexity cases; we expect further structural parallels with the curve graph in
finite type.

\subsection*{Questions}
This paper introduces $\Cnp(\Sigma)$ as a basic tool and uses it for coarse rank and divergence. Much
remains to be understood, and we conclude with several questions.

\begin{question}[Lower bounds for divergence]\label{q:lower-div}
When $\zeta(\Sigma)\ge 5$, is the divergence of $\Map(\Sigma)$ bounded \emph{below} by a quadratic
function? Equivalently, does there exist a pair of geodesic rays in $\Map(\Sigma)$ whose divergence
grows at least quadratically?
\end{question}
In the finite-type setting, quadratic lower bounds follow from strong contraction properties of
pseudo-Anosov axes \cite{DR}. It would be interesting to know whether appropriate analogues exist for
big mapping class groups in the range $\zeta(\Sigma)\ge 5$.

\begin{question}[When is $\Cnp(\Sigma)$ hyperbolic?]\label{q:Cnp-hyp-iff}
Find necessary and sufficient topological conditions on $\Sigma$ for $\Cnp(\Sigma)$ to be Gromov
hyperbolic.
\end{question}
When $\Sigma$ has finite genus, it is natural to introduce a combined complexity
\[
\xi(\Sigma):=3g(\Sigma)-3 +\zeta(\Sigma),
\]
which plays the role of the usual finite-type complexity $3g-3+n$. It is plausible that $\xi(\Sigma)$
is the correct parameter for an ``if and only if'' hyperbolicity statement for $\Cnp(\Sigma)$ in the
finite-genus setting, analogous to the finite-type case.

\begin{question}[Quadratic divergence beyond $\zeta\ge 5$]\label{q:Div-finite-genus}
Are there finite-genus surfaces with $\zeta(\Sigma)=4$ for which $\Map(\Sigma)$ has quadratic
divergence? More generally, can one characterize (topologically) when $\Map(\Sigma)$ has quadratic
divergence in terms of $\xi(\Sigma)$?
\end{question}

\begin{question}[Automorphisms of $\Cnp(\Sigma)$]\label{q:AutCnp}
Assume $\Sigma$ is stable, $\Map(\Sigma)$ is CB-generated, and $\zeta(\Sigma)\ge 5$. Is every
simplicial automorphism of $\Cnp(\Sigma)$ induced by an element of $\Map(\Sigma)$? Equivalently, is
$\Aut(\Cnp(\Sigma))\cong \Map(\Sigma)$ in this range?
\end{question}

\subsection*{Organization of the paper.}
In Section~\ref{sec:nondisplaceable} we recall the CB framework, the notion of stable surfaces, and the
anchor-surface technology from \cite{MR}, and we define the end-complexity $\zeta(\Sigma)$.
In Section~\ref{sec:nonperipheral-rank} we introduce peripheral and non-peripheral compact subsurfaces,
develop subsurface-projection length functions, and prove Theorem~\ref{introthm:Coarse} on infinite
coarse rank when $\zeta(\Sigma)\ge 4$.
In Section~\ref{sec:cnp} we define the graph $\Cnp(\Sigma)$, prove connectivity and infinite diameter,
and establish Gromov hyperbolicity, proving Theorem~\ref{introthm:Cnp}.
Finally, in Section~\ref{sec:divergence} we use $\Cnp(\Sigma)$ together with chains of commuting twists
and a ``persistent twist'' detour argument inspired by \cite{DR} to obtain the quadratic upper bound on
divergence (Theorem~\ref{introthmDiv}), and then deduce one-endedness (Corollary~\ref{introcor:one-ended}).

\subsection*{Acknowledgements}
Kasra Rafi is supported by NSERC Discovery grant RGPIN-05507. Yulan Qing is supported by Simons Foundation grant [SFI-MPS-TSM-00014066, Y.Q.]

\section{Preliminaries}

\subsection{Surfaces and mapping class group}
A \emph{surface} $\Sigma$ is a connected $2$-dimensional topological manifold, i.e., a second-countable Hausdorff $2$-dimensional space with no boundary. In this paper, we assume all surfaces to be orientable. The \emph{mapping class group} of $\Sigma$ is defined as the group $\Map(\Sigma)$ of all isotopy classes of orientation-preserving homeomorphisms of $\Sigma$. The group $\Map(\Sigma)$ is equipped with the quotient topology of the compact-open topology on the group $\Homeo^+(\Sigma)$ of all orientation-preserving homeomorphisms of $\Sigma$.

In this paper, a \emph{subsurface} $S$ of a surface $\Sigma$ is a connected closed subset of $\Sigma$ that is a manifold with boundary whose boundary consists of a finite number of pairwise non-intersecting simple closed curves, such that none of these boundary curves bounds a disk or a once-punctured disk in $\Sigma$. A surface $\Sigma$ is \emph{of finite type} if its fundamental group is finitely generated. We always assume that $\Sigma$ is of infinite type. Similarly, a subsurface $S$ of $\Sigma$ is of finite type if its fundamental group is finitely generated.

\subsection{The space of ends}
The space of ends of a surface $\Sigma$ is defined to be the inverse limit of the system of components of complements of compact subsets of $\Sigma$. Intuitively, each end corresponds to a way of leaving every compact subset of $\Sigma$ (see \cite{Ric} for details). Let $\End(\Sigma)$ be the end space of $\Sigma$ and $\End^g(\Sigma)\subseteq \End(\Sigma)$ be the subspace of $\End(\Sigma)$ consisting of non-planar ends. Also, let $\genus(\Sigma)$ be the genus of $\Sigma$ (possibly infinite). By a theorem of Richards \cite{Ric}, connected, orientable surfaces $\Sigma$ are classified up to homeomorphism by the triple $(\genus(\Sigma),\End(\Sigma),\End^g(\Sigma))$.

Given a subsurface $S\subseteq\Sigma$, the space of ends of $S$ is defined similarly and is denoted by $\End(S)$. The embedding of $S$ in $\Sigma$ gives a natural embedding of $\End(S)$ into $\End(\Sigma)$.

Every subsurface $S\subseteq\Sigma$ of finite type determines a finite partition $\Pi_S$ of the ends of $\Sigma$ where each part of the partition is the space of ends of a connected component of $\Sigma - S$. Given two subsets $X,Y\subseteq \End(\Sigma)$, we say that a subsurface $K$ \emph{separates} $X$ and $Y$ if $X$ and $Y$ belong to distinct subsets of the partition $\Pi_K$.

\begin{definition} \label{def:stable}
For $x\in \End(\Sigma)$, we call a neighborhood $U$ of $x$ \emph{stable} if for any smaller neighborhood $U'\subset U$ of $x$, there is a homeomorphism from $U$ to $U'$ fixing $x$. (See \cite[Proposition 3.2]{BDR} for equivalent definitions of a stable neighborhood). We say $\Sigma$ is \emph{stable} if every end $x \in \End(\Sigma)$ is stable.
\end{definition}

In this paper, we always assume the surface $\Sigma$ is stable. For an end $x \in \End(\Sigma)$, we denote the orbit of $x$ by
\[
E(x) = \{ \phi(x) \mid \phi \in \Map(\Sigma) \}.
\]
We say $x$ and $x'$ are of the same type if $x' \in E(x)$. This defines a partial order on the set of ends as follows. We say $x \preceq y$ if $E(x)$ accumulates to $y$. It was shown in \cite{MR} that $x \preceq y$ and $y \preceq x$ implies $x \in E(y)$.

Let $\calM(\End(\Sigma))$ be the set of maximal ends, that is, the set of points $x \in \End(\Sigma)$ such that if $x \preceq y$ then $y \in E(x)$. For a stable surface, $\calM(\End(\Sigma))$ is non-empty and consists of finitely many different types. We refer to elements of $\calM(\End(\Sigma))$ as \emph{maximal ends}. For every maximal end $x$, either $E(x)$ is finite or it is a Cantor set. We refer to these as \emph{maximal ends of finite type} and \emph{maximal ends of Cantor type}, respectively.

\subsection{Curves and curve graphs}
A simple closed curve $\alpha$ on a surface $\Sigma$ is a free homotopy class of an essential simple closed curve. Here, \emph{simple} means $\alpha$ has a representative that does not self-intersect and \emph{essential} means that $\alpha$ does not bound a disk or a once-punctured disk. For a subsurface $S$ of $\Sigma$, a curve $\alpha$ is \emph{in} $S$ if it has a representative that is contained in $S$ and if $\alpha$ is essential in $S$, that is, $\alpha$ is not parallel to a boundary component of $S$.

The curve graph of $\Sigma$, denoted by $\calC(\Sigma)$, is defined \cite{harvey} to be the graph whose vertex set is the set of curves in $\Sigma$ and whose edges are pairs of distinct curves in $\Sigma$ that have disjoint representatives. Similarly, for a finite type surface $S$ of $\Sigma$, let the curve graph of $S$, $\calC(S)$, be the graph whose vertex set is the set of curves in $S$ and whose edges are pairs of distinct curves in $S$ that have disjoint representatives. Then $\calC(S)$ is an induced subgraph of the curve graph $\calC(\Sigma)$.

We equip the curve graph with a metric where every edge has length 1. We denote the associated distance by $d_\Sigma$ and $d_S$, that is, for curves $\alpha$ and $\beta$ in $S$, $d_S(\alpha, \beta)$ is the smallest number $n$ such that there is a sequence
\[
\alpha = \alpha_0, \alpha_1 \dots, \alpha_n = \beta
\]
where $\alpha_i$ and $\alpha_{i+1}$ are disjoint for $i=0, \dots, n-1$.

Recall that, for finite type surfaces $S$, the curve graph $\calC(S)$ has infinite diameter \cite{harvey} and is Gromov hyperbolic \cite{MM2}, and the mapping class group acts on $\calC(S)$ by isometries with unbounded orbits \cite{harvey}. For infinite type surfaces, the curve graph (defined as in the case of finite type surfaces) has bounded diameter \cite{BDR}.

\subsection{Markings on a surface}
We say a set of curves $\{\alpha_i\}$ in $\Sigma$ \emph{fill} a subsurface $S$ if $\alpha_i$ are contained in $S$ and for every curve $\beta$ in $S$, there exists a curve $\alpha_i$ such that
\[\beta \cap \alpha_i \neq \emptyset.\]

A \textit{marking} $\mu_S$ of a finite-type surface $S$ is a collection of simple closed curves
\[
\mu_S=\{\alpha_1,\hdots,\alpha_n, \beta_1,\hdots,\beta_n\}
\]
such that:
\begin{itemize}
 \item The curves $\{\alpha_i\}$ form a pants decomposition of $S$.
 \item $\alpha_i \cap \beta_j = \emptyset$ whenever $i\neq j$.
 \item For any $i$, either $\alpha_i$ and $\beta_i$ fill a punctured torus and $|\alpha_i\cap \beta_i| = 1$, or they fill a four-times punctured sphere and $|\alpha_i\cap \beta_i| = 2$.
\end{itemize}

It follows from the definition that a marking $\mu_S$ fills the subsurface $S$. See \cite{MM2} for more details.

\subsection{Coarse geometry of big mapping class groups}\label{sec:nondisplaceable}
Let $G$ be a Polish topological group. A subset $ A \subset G$ is coarsely bounded, abbreviated CB, if every compatible left-invariant metric on $G$ gives $A$ finite diameter. CB sets behave in many ways like finite sets and, by work of Rosendal \cite{Rosendal}, the theory of uncountable groups that are generated by CB sets resembles the theory of finitely generated groups.

\begin{theorem}[\cite{Rosendal}]\label{thm:Rosendall}
Let $G$ be a Polish group that has both a CB neighborhood of the identity and is generated by a CB subset. Then the identity map is a quasi-isometry between $G$ endowed with any two word metrics associated to symmetric, CB generating sets. We then say $G$ is a CB-generated group.
\end{theorem}

In \cite{MR}, Mann and Rafi gave descriptions of stable surfaces $\Sigma$ where $\Map(\Sigma)$ has a CB neighborhood of the identity and surfaces where $\Map(\Sigma)$ is generated by a CB subset. We recall/summarize these statements here (see Proposition 5.5 and Theorem 5.7 in \cite{MR}).

Since the topology of $\Map(\Sigma)$ comes from the compact-open topology on $\Homeo^+(\Sigma)$, a neighborhood of the identity in $\Map(\Sigma)$ can be described as the stabilizer of a compact subsurface $K$. Following \cite{MR}, we define
\begin{equation} \label{eq:VK}
\calV_K = \big\{ \phi \in \Map(\Sigma) \mid \phi(K)=K \quad \text{and} \quad \phi|_K = \text{id} \big\}.
\end{equation}
Then every end of $\Sigma$ is an end of some component of $\Sigma-K$, that is, $K$ partitions $\End(\Sigma)$ into finitely many disjoint subsets. We say a component $S$ of $\Sigma - K$ contains an end $x \in \End(\Sigma)$ if $x \in \End(S)$.

For $\calV_K$ to be CB, $K$ has to be large enough. This can be made precise by examining the way $K$ decomposes $\End(\Sigma)$. Essentially, we need each component of $\Sigma - K$ to contain ends of at most one maximal type. Also, if $x \in \calM(\End(\Sigma))$ and $E(y)$ accumulates to $x$, we need $E(y)$ also to have an accumulation point in some component of $\Sigma - K$ other than the one containing $x$. In particular, for $x \in \calM(\End(\Sigma))$, if $E(x)$ is a Cantor set, since $E(x)$ accumulates to itself, we need there to be at least two components of $\Sigma - K$ that contain points of $E(x)$. (This is to ensure part (iii) of Theorem 5.7 in \cite{MR} holds).

In more detail, there has to exist a surface of finite type $K$ such that the connected components of $\Sigma - K$ partition $\End(\Sigma)$ as
\[
\End(\Sigma) = \bigsqcup_{A \in \mathcal{A}} A \,\,\sqcup\,\, \bigsqcup_{P \in \mathcal{P}} P,
\]
such that the following holds:
\begin{enumerate}
    \item[(i)] Each connected component of $\Sigma - K$ has one or infinitely many ends and zero or infinite genus.
    \item[(ii)] For $A \in \mathcal{A}$, the set of maximal points $\mathcal{M}(A)$ is either a single point or a Cantor set; points in $\mathcal{M}(A)$ are all of the same type and they are maximal in $\End(\Sigma)$. Furthermore,
    \[
    \mathcal{M}(\End(\Sigma)) = \bigsqcup_{A \in \mathcal{A}} \mathcal{M}(A).
    \]
    \item[(iii)] For each $P \in \mathcal{P}$, there exists some $A \in \mathcal{A}$ such that $P$ is homeomorphic to a clopen subset of $A$.
    \item[(iv)] For $y \in \End(\Sigma)$, if $E(y)$ has an accumulation point in $A$, it also has an accumulation point outside of $A$.
\end{enumerate}

We refer to a subsurface $K$ with the above properties as an \emph{anchor surface}. In fact, we fix a minimal anchor surface $K_0$, meaning an anchor surface where the genus and the number of connected components of $\Sigma-K_0$ are minimal.

Let us examine how this minimality can be achieved. If $\Sigma$ has finite genus, then the genus of $K_0$ has to equal the genus of $\Sigma$ because, in this case, every component of $\Sigma- K_0$ has to be genus zero. Otherwise (meaning if the genus of $\Sigma$ is zero or infinite), $K_0$ has genus zero since all the genus can be pushed near some non-planar end of $\Sigma$.

Furthermore, as mentioned before, the set $\calM(\End(\Sigma))$ has finitely many types; some of them are isolated in $\calM(\End(\Sigma))$ and some are Cantor types. In the minimal case, each set of Cantor type maximal ends should appear in exactly two sets in $\calA$ and each isolated point is contained in its own component. Hence, the minimum value for $|\calA|$ is
\begin{equation} \label{eq:A}
|\calA| = 2 \cdot \big(\#\, \text{of Cantor types in } \calM(\End(\Sigma))\big) +
\big(\# \, \text{of isolated points in } \calM(\End(\Sigma))\big).
\end{equation}
If $E(y)$ accumulates to $x \in \calM(\End(\Sigma))$ where $E(x)$ is a Cantor set, then $E(y)$ accumulates to every point in $E(x)$. Hence $E(y)$ has accumulation points in at least two sets in $\calA$. But it is possible that $y$ is not maximal and $E(y)$ accumulates to a single isolated maximal point $x$ which is contained in $A \in \calA$. The set $E(y)$ has to have an accumulation point outside of $A$, but it may not have an accumulation point in any other $A' \neq A$. This could happen, for example, if $\overline{E(y)}$ (the closure of $E(y)$) is a Cantor set and
\[
\overline{E(y)} = E(y) \cup \{ x\}.
\]
Then $E(y)$ accumulates to itself. Or $E(z)$ could accumulate to points in $E(y)$, where $E(y)$ is as above. In this case, $E(y)$ has to appear in some $P \in \calP$. In this case, we say $y$ \emph{uniquely accumulates to the maximal end $x$} or we say $x$ is the \emph{unique maximal accumulation point of $y$}.

Assume $x \in A$ is the unique maximal accumulation point of $E(y)$ and the other (non-maximal) accumulation point of $E(y)$ is in $P$, and that $x' \in A'$ is the unique maximal accumulation point of $E(y')$ and the other accumulation point of $E(y')$ is in $P'$. Then
\[
E(y) \cap A' = \emptyset, \qquad\text{and} \qquad
E(y') \cap A = \emptyset.
\]
Also,
\[
E(y) \cap P' = \emptyset, \qquad\text{and} \qquad
E(y') \cap P = \emptyset,
\]
because if $P$ intersects both $E(y)$ and $E(y')$ then it cannot be contained in any $A'' \in \calA$.
This means, in the minimal case,
\begin{equation}\label{eq:P}
|\calP| =\# \big(\text{isolated maximal ends that are unique maximal accumulation points}\big).
\end{equation}
To summarize, for every maximal end of Cantor type $x$, we divide $E(x)$ between sets $A_x^1$ and $A_x^2$ such that $A_x^1 \sqcup A_x^2$ contains a stable neighborhood of every point in $E(x)$. If $x \in \calM(\End(\Sigma))$ is isolated, we choose a stable neighborhood $A_x$ of $x$. We let $\calA$ be the collection of these sets. For every isolated maximal end $x \in A_x$ that is a unique maximal accumulation point, we form a set $P_x$ containing representatives of all accumulation points of sets $E(y)$ where $E(y)$ uniquely accumulates to the maximal end $x$. Then $\calP$ is the collection of such sets $P_x$. Every other remaining point is not maximal and can be added to some $A \in \calA$ or $P \in \calP$.

\begin{definition}\label{countcomplexity}
The end-complexity $\zeta(\Sigma)$ of an infinite type surface $\Sigma$ is the minimum number of boundary components of an anchor surface. That is,
\begin{align*}
\zeta(\Sigma) &= 2 \cdot \# \big(\text{$E(x)$ where $E(x)\subset \calM(\End(\Sigma))$ is a Cantor set}\big) + \\
& + \# \big( \text{of isolated points in $\calM(\End(\Sigma))$}\big) +\\
&+ \# \big(\text{isolated maximal ends that are unique maximal accumulation points}\big).
\end{align*}
\end{definition}

\begin{example} We now examine this definition in a specific example. 
Let $\Sigma$ be the surface depicted in \figref{fig:5}. The set of maximal ends $\calM(\Sigma)$ 
consists of two isolated points $x_A$ and $x_C$ and a Cantor set of non-planar points. 
The point $x_A$ is an accumulation of point of two different Cantor types, 
$E(y')$ that are accumulated by punctures and $E(y)$ that are not. Since the points
$E(y)$ and $E(y')$ uniquely accumulate to $x_A$ (they do not accumulate to any other maximal end)
the anchor surface $K_0$ has to separate some of $E(y)$ and $E(y')$ from $A$ and place
them in a set $P$. Also, $K_0$ has to separate the Cantor set of non-planar points into two sets 
$B \sqcup B'$. Since $x_C$ is isolated in $\End(\Sigma)$ and the set $C=\{ x_C\}$ contains only 
one point.
    
\begin{figure}[ht]
\setlength{\unitlength}{0.01\linewidth}
\begin{picture}(100, 55) 
\put(13,0){
    \includegraphics[width=0.7\linewidth]{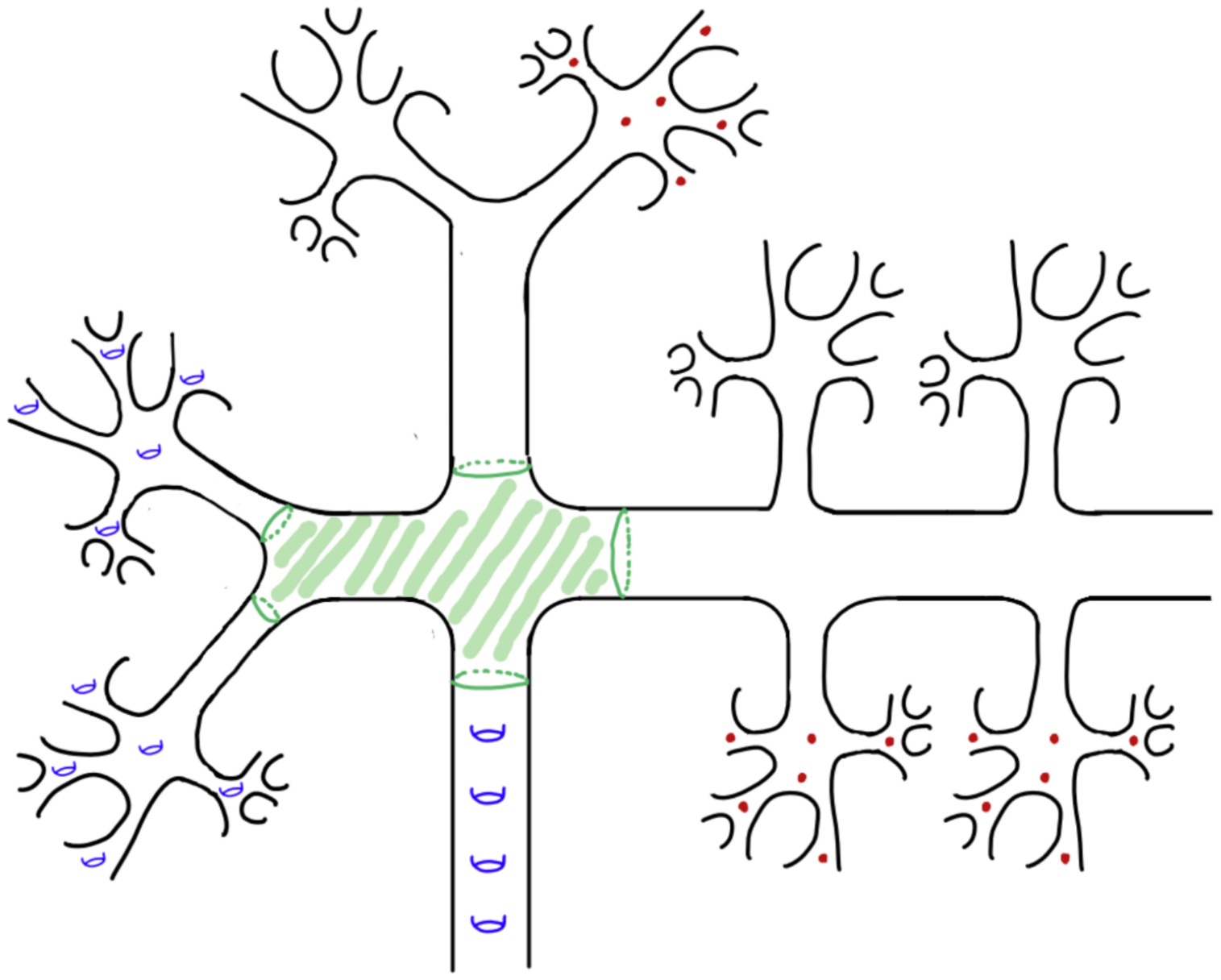}
}
\put(84, 23) {$x_A$}
\put(80, 43) {$A$}
\put(10,  10) {$B'$}
\put(10,  30) {$B$}
\put(16,  6) {$x_{B'}$}
\put(16,  37) {$x_B$}
\put(28,  41) {$y$}
\put(58,  48) {$y'$}
\put(40,  53) {$P$}
\put(40.5,  -1) {$x_C$}
\put(40.5,23) {$K_0$}
\end{picture}
    \caption{A surface $\Sigma$ with $\zeta(\Sigma) = 5$.}
    \label{fig:5}   
\end{figure}

That is, $\calM(\Sigma)$ contains two isolated maximal ends and a Cantor set of maximal ends, and one of the isolated maximal ends is a unique accumulation point. Hence $\zeta(\Sigma) = 5$. 
\end{example}

So far, we have produced a CB neighborhood of the identity $\calV_{K_0}$ in $\Map(\Sigma)$. To find a CB set that generates $\Map(\Sigma)$, we need further assumptions.

\begin{theorem}[\cite{MR}, Theorem 1.6] \label{thm:MR_class}
(Classification of CB generated mapping class groups). For a stable surface $\Sigma$ with locally (but not globally) CB mapping class group, $\Map(\Sigma)$ is CB generated if and only if $\End(\Sigma)$ is finite rank and not of limit type.
\end{theorem}

The definition of terms in the above theorem is not relevant to us and we refer the reader to \cite{MR}. Instead, we give a description of a CB generating set assuming $\End(\Sigma)$ is finite rank and not of limit type.

\begin{lemma}[\cite{MR}, Lemma 6.10]
Assume that $\Map(\Sigma)$ is locally CB and that $\End(\Sigma)$ does not have limit type. Then:
\begin{itemize}
\item For every $A \in \calA$, there is a point $x_A \in \calM(A)$ and a neighborhood $N(x_A)\subset A$ containing $x_A$ such that $A-N(x_A)$ contains a representative of every type in $A-\{x_A\}$.
\item For every pair $A,B \in \calA$, there is a clopen set $W_{A,B} \subset A- N(x_A)$ with the property that for any end $z$, $E(z) \cap W_{A,B} \neq \emptyset$ if and only if
\[
E(z) \cap (A - \{ x_A \}) \neq \emptyset \qquad\text{and}\qquad E(z) \cap (B - \{ x_B \}) \neq \emptyset.
\]
\item For every $A \in \calA$, there is a clopen set $W_A \subset (A -N(x_A))$ with the property that for any end $z$, if $E(z) \cap A - \{x_A\} \neq \emptyset$ and, for all $B \neq A$, $E(z) \cap B = \emptyset$, then $E(z) \cap W_A \neq \emptyset$.
\item The sets $W_{A,B}$ and $W_A$ give a decomposition of $A- N(x_A)$:
\[
A- N(x_A) = \left(\bigsqcup_{B \in (\calA - \{A\})} W_{A,B}\right) \sqcup \left(\bigsqcup_{A \in \calA} W_A\right).
\]
\end{itemize}
\end{lemma}

These sets $W_{A,B}$ will allow us to define some shift maps between $A$ and $B$. It follows from \cite[Section 6.4]{MR} that the self-similarity of stable neighborhoods allows us to extend these sets into infinite sequences. Since $A$ is stable and $x_A \in A$, the space $A - \{x_A\}$ consists of infinitely many disjoint copies of the fundamental domain $A - N(x_A)$. Similarly, if $P \in \mathcal{P}$ is associated to $A$, then $P$ is homeomorphic to a clopen subset of $A - N(x_A)$ (by property (iii) of the anchor surface).

We can therefore decompose the entire end space (excluding the maximal ends $x_A$) into disjoint orbits (see also \cite[Proposition 3.2]{BDR}).

\begin{lemma}[Decomposition of Ends] \label{lem:ends}
There exists a decomposition of the space of ends
\[
\End(\Sigma) - \{ x_A \mid A \in \calA \} = \left(\bigsqcup_{A \in \calA} \bigsqcup_{B \in \calA - \{A\}} \mathcal{O}_{A,B}\right) \sqcup \left(\bigsqcup_{A \in \calA} \mathcal{O}_A\right),
\]
where the sets $\mathcal{O}_{A,B}$ and $\mathcal{O}_A$ are unions of disjoint clopen sets indexed by integers:
\begin{itemize}
    \item For every pair of distinct $A, B \in \calA$, the set $\mathcal{O}_{A,B}$ is a disjoint union
    \[
    \mathcal{O}_{A,B} = \bigsqcup_{k \in \mathbb{Z}} W_{A,B}^k,
    \]
    where each $W_{A,B}^k$ is homeomorphic to the set $W_{A,B}$ defined in Lemma 6.10. Furthermore, we can index these sets such that for $k \leq 0$, $W_{A,B}^k \subset A$, and for $k > 0$, $W_{A,B}^k \subset B$.

    \item For every $A \in \calA$, the set $\mathcal{O}_A$ is a disjoint union
    \[
    \mathcal{O}_A = \bigsqcup_{k \in \mathbb{Z}} W_A^k,
    \]
    where each $W_A^k$ is homeomorphic to $W_A$. Furthermore, for $k \leq 0$, $W_A^k \subset A$, and for $k > 0$, the sets $W_A^k$ cover the ends in the sets $P \in \mathcal{P}$ associated with $A$.
\end{itemize}
\end{lemma}

With this decomposition fixed, we can now define the elements of our generating set.

\subsection{The generating set}\label{generatingset}
We define a finite set of mapping classes $\mathcal{G}$ consisting of shift maps, maximal end permutations, and local generators.

\subsubsection*{Shift maps}
The decomposition above provides natural tracks along which we can shift the ends of the surface.
\begin{itemize}
    \item \textbf{Inter-region shifts ($f_{A,B}$):} For every distinct pair $A, B \in \calA$, let $f_{A,B}$ be a homeomorphism supported on a subsurface containing $\mathcal{O}_{A,B}$ that satisfies:
    \[
    f_{A,B}(W_{A,B}^k) = W_{A,B}^{k+1} \quad \text{for all } k \in \mathbb{Z},
    \]
    and is the identity everywhere else. This map effectively shifts a copy of the shared types $W_{A,B}$ from the region $A$ into the region $B$.

    \item \textbf{Unique accumulation shifts ($f_A$):} For every $A \in \calA$ such that $W_A$ is non-empty (i.e., there are $P$ sets associated with $A$), let $f_A$ be a homeomorphism supported on a subsurface containing $\mathcal{O}_A$ that satisfies:
    \[
    f_A(W_A^k) = W_A^{k+1} \quad \text{for all } k \in \mathbb{Z},
    \]
    and is the identity everywhere else. This map shifts ends from $A$ into the corresponding sets $P \in \calP$.
\end{itemize}

\subsubsection*{Maximal end permutations}
Let $\text{Sym}(\calA)$ be the group of permutations of the set $\calA$. We consider the subgroup of permutations $\sigma$ such that for all $A \in \calA$, the maximal end $x_A$ is of the same topological type as $x_{\sigma(A)}$.
For each such $\sigma$, we fix a homeomorphism $h_\sigma$ such that:
\[
h_\sigma(A) = \sigma(A) \quad \text{and} \quad h_\sigma(x_A) = x_{\sigma(A)}.
\]
We include a finite set of such maps generating this subgroup.

\subsubsection*{Local generators}
Finally, we need to generate the mapping class group of the anchor surface. Let $K \supseteq K_0$ be a compact subsurface of finite type large enough to contain the boundaries of the shift maps' supports restricted to the "0-level" (specifically, the boundaries separating $W^0$ and $W^1$ for all shift maps).
Let $\mathcal{L}$ be a standard finite generating set for $\Map(K)$, consisting of:
\begin{itemize}
    \item Dehn twists along a finite set of simple closed curves in $K$.
    \item If $K$ has boundary components that are permutable, half-twists braiding these boundary components.
\end{itemize}

\begin{theorem} \label{thm:generating-set}
Let $\Sigma$ be a stable surface of infinite type and let $K_0$ be the minimal anchor surface we fixed above. Then the set
\[
\mathcal{G} = \{ f_{A,B} \mid A,B \in \calA, A \neq B \}
\cup \{ f_A \mid A \in \calA \} \cup \{ h_\sigma \} \cup \mathcal{L}
\cup \calV_{K_0}
\]
is a coarsely bounded (CB) generating set for $\Map(\Sigma)$.
\end{theorem}

\section{Non-peripheral subsurfaces and the coarse rank of the mapping class group}\label{sec:nonperipheral-rank}

In this section we introduce the notion of \emph{non-peripheral} compact
subsurfaces of an infinite-type surface $\Sigma$, and use them to produce
quasi-flats in $\Map(\Sigma)$ when $\Map(\Sigma)$ is CB-generated (see
Section~\ref{sec:nondisplaceable}). The main result is
Theorem~\ref{thm:infinite-rank}, which states that, when $\zeta(\Sigma) \geq 4$, $\Map(\Sigma)$ has infinite
coarse rank.

\subsection{Peripheral and non-peripheral subsurfaces}\label{subsec:nonperipheral}

For finite-type surfaces, a curve is called peripheral if it bounds a once-punctured disk (or equivalently, if it can be pushed into a cusp). For infinite-type surfaces there is no canonical finite set of cusps. Our definition is a direct analogue: a compact subsurface is peripheral if it can be pushed away from \emph{every} compact region.

\begin{definition}[Peripheral and non-peripheral subsurfaces]\label{def1}
A compact subsurface $R\subset \Sigma$ is \emph{peripheral} if for every compact
subsurface $K\subset \Sigma$ there exists $g\in \Map(\Sigma)$ such that
\[
g(R)\cap K=\emptyset.
\]
Otherwise, $R$ is \emph{non-peripheral}.
\end{definition}

Recall that in Section~\ref{sec:nondisplaceable} we fixed a minimal anchor surface
$K_0$ (a compact subsurface of finite type) with the properties described there.
The next lemma shows that to test peripheralness it suffices to test disjointness
from $K_0$.

\begin{lemma}\label{lem:intersect-anchor}
A compact subsurface $R\subset \Sigma$ is peripheral if and only if there exists
$g\in \Map(\Sigma)$ such that
\[
g(R)\cap K_0=\emptyset.
\]
\end{lemma}

\begin{proof}
If $R$ is peripheral, apply Definition~\ref{def1} with $K=K_0$ to obtain some
$g\in\Map(\Sigma)$ with $g(R)\cap K_0=\emptyset$.

Conversely, suppose there exists $g_0\in\Map(\Sigma)$ such that $g_0(R)\cap K_0=\emptyset$.
Let $K\subset \Sigma$ be any compact subsurface. We will find $h\in\Map(\Sigma)$ such that
$h(R)\cap K=\emptyset$.

Since $K_0$ is an anchor surface (Section~\ref{sec:nondisplaceable}), the components of
$\Sigma-K_0$ correspond to the clopen sets in the partition
\[
\End(\Sigma)=\bigsqcup_{A\in\calA}A\ \sqcup\ \bigsqcup_{P\in\calP}P
\]
from Section~\ref{sec:nondisplaceable}. Let $\Sigma_A$ be the component of $\Sigma-K_0$
containing $g_0(R)$, where $A\in\calA$ is the corresponding end space.
By Lemma~\ref{lem:ends} (Section~\ref{generatingset}) and the definition of the
shift maps, there exists a mapping class $s\in\Map(\Sigma)$
which is a word in the shift maps $\{f_A\}$ and $\{f_{A,B}\}$ and has the following property:
for a fixed stable neighborhood $N(x_A)\subset A$ (as in Lemma~6.10 of \cite{MR} recalled in
Section~\ref{sec:nondisplaceable}), we have
\[
\End\big(s^m(\Sigma_A)\big)\ \subset\ N(x_A)
\quad\text{for all sufficiently large }m,
\]
and hence $s^m(\Sigma_A)$ eventually leaves every compact subset of $\Sigma$.
In particular, there exists $m$ such that $s^m(\Sigma_A)\cap K=\emptyset$.

Set $h=s^m\circ g_0$. Then $h(R)\subset s^m(\Sigma_A)$, so $h(R)\cap K=\emptyset$.
Since $K$ was arbitrary, Definition~\ref{def1} implies that $R$ is peripheral.
\end{proof}

Thus we obtain a convenient equivalent characterization.

\begin{definition}[Non-peripheral subsurfaces]\label{def2}
A compact subsurface $R\subset \Sigma$ is non-peripheral if and only if for every
$g\in \Map(\Sigma)$ we have
\[
g(R)\cap K_0\neq \emptyset.
\]
\end{definition}

\begin{remark}\label{rem:nondisplaceable}
Recall from \cite{MR} that a compact subsurface $R$ is \emph{non-displaceable} if
$R\cap g(R)\neq \emptyset$ for every $g\in\Map(\Sigma)$ (see also
Section~\ref{sec:nondisplaceable}). Non-displaceable subsurfaces are a basic source of
unbounded coarse geometry for big mapping class groups (e.g.\ \cite[Proposition 2.8]{MR}).
Every non-displaceable subsurface is non-peripheral, but non-peripheral subsurfaces
form a larger class and will be more flexible for producing quasi-flats.
\end{remark}

\subsection{Coarse rank and quasi-flats}\label{Sec:rank}

As discussed in Section~\ref{sec:nondisplaceable}, a CB-generated Polish group
admits a well-defined quasi-isometry type of word metrics coming from symmetric
CB generating sets \cite{Rosendal}. Following \cite{GRV}, we use this to define
a rank notion for big mapping class groups.

\begin{definition}[Coarse rank {\cite{GRV}}]\label{def:coarse-rank}
Assume $\Map(\Sigma)$ is CB-generated (Section~\ref{sec:nondisplaceable}). Equip
$\Map(\Sigma)$ with a word metric $d_{\mathcal{G}}$ associated to any symmetric
CB generating set $\mathcal{G}$ (e.g.\ the set from Theorem~\ref{thm:generating-set}).
The \emph{coarse rank} of $\Map(\Sigma)$ is the largest $n\ge 0$ for which there exists a
quasi-isometric embedding $\ZZ^n \hookrightarrow (\Map(\Sigma),d_{\mathcal{G}})$.
If such $n$ is unbounded, we say $\Map(\Sigma)$ has \emph{infinite coarse rank}.
\end{definition}

\subsection{Length functions and subsurface projection distances}\label{lf}

A convenient way to obtain lower bounds on word metrics is via length functions.

\begin{definition}\label{Def:length_function}
Let $G$ be a topological group. A \emph{length function} on $G$ is a function
$\ell\colon G\to [0,\infty)$ such that
\begin{itemize}
    \item $\ell(\id)=0$;
    \item $\ell(g)=\ell(g^{-1})$ for all $g\in G$;
    \item $\ell(gh)\le \ell(g)+\ell(h)$ for all $g,h\in G$.
\end{itemize}
\end{definition}

\begin{remark}\label{rem:Rosendal-bounded}
By \cite[Proposition 2.7(5)]{Rosendal}, every length function on a Polish group is
bounded on every coarsely bounded (CB) subset (see Section~\ref{sec:nondisplaceable}).
In particular, if $\mathcal{G}$ is a CB generating set for $G$, then
$\sup_{s\in \mathcal{G}}\ell(s)<\infty$ and $\ell$ gives a uniform lower bound for the word metric.
\end{remark}

We now recall subsurface projections and the associated projection distances in curve
graphs. These constructions go back to Masur--Minsky and are standard; we follow
\cite{MM2}. Curve graphs and their metrics were defined in Section~\ref{sec:nondisplaceable}.

\subsection*{Subsurface projections and projection distances}
In this subsection we recall the coarse definition of the projection $\pi_S$ to a
non-annular subsurface $S$ and the induced projection distance.

Let $S$ be a finite-type subsurface of $\Sigma$ (possibly an annulus). For a curve
$\alpha$ on $\Sigma$ (Section~\ref{sec:nondisplaceable}), the \emph{subsurface
projection} $\pi_S(\alpha)$ is defined whenever $\alpha$ intersects $S$ essentially.

\begin{definition}[Subsurface projection for non-annular subsurfaces]\label{def:subsurface-projection}
Assume $S$ is a finite-type subsurface which is \emph{not} an annulus.
If a curve $\alpha$ has a representative intersecting $S$ essentially, put $\alpha$
and $\partial S$ in minimal position and consider the collection of essential arcs
of $\alpha\cap S$. For each such arc, perform surgery with $\partial S$ (equivalently:
take the boundary components of a regular neighborhood of the arc together with $\partial S$)
to obtain a finite set of essential curves in $S$. Define $\pi_S(\alpha)\subset \calC(S)$
to be the union of these curves. If $\alpha$ is disjoint from $S$ (or $\alpha\cap S$ has no
essential arc component), we set $\pi_S(\alpha)=\emptyset$.

If $P$ is a multicurve (e.g.\ $P=\partial R$ for some subsurface $R$), define
\[
\pi_S(P)\ :=\ \bigcup_{\alpha\in P}\pi_S(\alpha),
\]
and if $\mu$ is a marking on a finite-type surface (Section~\ref{sec:nondisplaceable}), define
\[
\pi_S(\mu)\ :=\ \bigcup_{\gamma\in \mu}\pi_S(\gamma),
\]
where the union is over all base curves and transversals of $\mu$.
\end{definition}

\begin{remark}\label{rem:projection-coarse}
Different choices of representatives and surgeries change $\pi_S(\alpha)$ only by a uniformly
bounded amount in $\calC(S)$; in particular $\diam_{\calC(S)}(\pi_S(\alpha))$ is uniformly
bounded (depending only on the topological type of $S$). Thus subsurface projections are
\emph{coarsely well-defined}. See \cite{MM2} for details.
\end{remark}

\begin{definition}[Projection distance]\label{def:projection-diameter}
Let $S$ be a finite-type subsurface which is not an annulus, and let $X,Y$ be curves,
multicurves, or markings. If $\pi_S(X)\neq \emptyset$ and $\pi_S(Y)\neq \emptyset$, define
\begin{equation}\label{eq:projection-diam}
d_S(X,Y)\ :=\ \diam_{\calC(S)}\big(\pi_S(X)\cup \pi_S(Y)\big)
 = \sup\big\{d_S(x,y)\mid x\in \pi_S(X),\ y\in \pi_S(Y)\big\}.
\end{equation}
If $\pi_S(X)=\emptyset$ or $\pi_S(Y)=\emptyset$, then the distance is not defined.
\end{definition}

\subsection*{Annular projections and twisting}
In this subsection we recall annular projections and the associated twisting distance.

When $S$ is an annulus with core curve $\gamma$, subsurface projection is taken to the
\emph{annular arc graph} $\calC(\gamma)$, whose vertices are isotopy classes (rel.\ endpoints)
of essential arcs in the annular cover $\widetilde S_\gamma$ corresponding to $\gamma$, with endpoints
on $\partial \widetilde S_\gamma$, and whose edges join disjoint arcs. For a curve $\alpha$ intersecting
$\gamma$, the projection $\pi_\gamma(\alpha)$ is the set of lifts of $\alpha$ to $\widetilde S_\gamma$
that are arcs connecting the two components of $\partial \widetilde S_\gamma$.
For a multicurve $P$ and a marking $\mu$, define $\pi_\gamma(P)$ and $\pi_\gamma(\mu)$ by taking unions
over constituent curves as in Definition~\ref{def:subsurface-projection}. Again, see \cite{MM2} for details.

\begin{remark}\label{rem:annular-coarse}
As in the non-annular case, $\pi_\gamma(\alpha)$ is coarsely well-defined and has uniformly
bounded diameter in $\calC(\gamma)$. In particular, the distance between two projections is
well-defined up to a bounded additive error.
\end{remark}

\begin{definition}[Relative twisting / annular distance]\label{def:Twisting}
Let $\gamma$ be a curve and let $\calC(\gamma)$ be the annular arc graph.
For curves, multicurves, or markings $X,Y$ with $\pi_\gamma(X),\pi_\gamma(Y)\neq \emptyset$,
define the \emph{relative twisting} (annular projection distance) about $\gamma$ by
\[
\mathrm{tw}_\gamma(X,Y)\ :=\ d_{\gamma}\big(\pi_\gamma(X),\pi_\gamma(Y)\big)
\ :=\ \diam_{\gamma}\big(\pi_\gamma(X)\cup \pi_\gamma(Y)\big).
\]
If either projection is empty, we set $\mathrm{tw}_\gamma(X,Y)=0$.
\end{definition}

\begin{remark}\label{rem:twist-translation}
The main fact we will use about relative twisting is that a Dehn twist $T_\gamma$ acts by a translation
on $\calC(\gamma)$, and hence $\mathrm{tw}_\gamma\big(X,T_\gamma^n(X)\big)$ grows coarsely linearly in $|n|$
for any $X$ with $\pi_\gamma(X)\neq \emptyset$.
\end{remark}

\subsection{A subsurface-projection length function}\label{subsec:LcalR}

We now package the above projection distances into a length function on $\Map(\Sigma)$.

Let $\calR=\{R_1,\dots,R_n\}$ be a collection of pairwise disjoint non-peripheral
compact subsurfaces of $\Sigma$. By Lemma~\ref{lem:intersect-anchor} and
Definition~\ref{def2}, every translate $g(R_i)$ intersects $K_0$ essentially.

\begin{definition}\label{Def:lengthfunctionR}
Let $\calR=\{R_1,\dots,R_n\}$ be disjoint non-peripheral compact subsurfaces of $\Sigma$.
Fix a marking $\mu$ on $K_0$ (Section~\ref{sec:nondisplaceable}). Define
$L_{\calR}^{\mu}\colon \Map(\Sigma)\to [0,\infty)$ by
\begin{equation}\label{eq:LcalR-def}
L_{\calR}^{\mu}(f)\ :=\ \sup_{g\in \Map(\Sigma)}\ \sum_{i=1}^n
d_{g(R_i)}\big(\mu,f(\mu)\big),
\end{equation}
when $f$ is not the identity. For the identity element we let $L_{\calR}^{\mu}(\id)=0$.
We suppress $\mu$ from the notation and write $L_{\calR}$ when $\mu$ is fixed.
\end{definition}

\begin{lemma}\label{Lem:islengthfunction}
$L_{\calR}$ is a length function in the sense of Definition~\ref{Def:length_function}.
\end{lemma}

\begin{proof}
Nonnegativity and $L_{\calR}(\id)=0$ are immediate from the definition. For symmetry, we use $\Map(\Sigma)$--equivariance of projections (and of annular
projections) to observe that for any finite-type subsurface $Y$ and any $f\in \Map(\Sigma)$,
\[
d_{Y}\big(\mu,f^{-1}(\mu)\big)\ =\ d_{f(Y)}\big(\mu,f(\mu)\big).
\]
Therefore
\begin{align*}
L_{\calR}(f^{-1})
&=\sup_{g}\sum_i d_{g(R_i)}\big(\mu,f^{-1}(\mu)\big)
=\sup_{g}\sum_i d_{f g(R_i)}\big(\mu,f(\mu)\big)\\
&=\sup_{h}\sum_i d_{h(R_i)}\big(\mu,f(\mu)\big)
= L_{\calR}(f),
\end{align*}
where we reparametrized $h=fg$.

For subadditivity, fix $f,h\in \Map(\Sigma)$ and apply the triangle inequality in each
$\calC(g(R_i))$ (or annular arc graph):
\[
d_{g(R_i)}\big(\mu,fh(\mu)\big)\ \le\
d_{g(R_i)}\big(\mu,f(\mu)\big)\ +\ d_{g(R_i)}\big(f(\mu),fh(\mu)\big).
\]
Summing over $i$ and taking suprema gives
\[
L_{\calR}(fh)\ \le\ \sup_{g}\sum_i d_{g(R_i)}\big(\mu,f(\mu)\big)
\ +\ \sup_{g}\sum_i d_{g(R_i)}\big(f(\mu),fh(\mu)\big).
\]
Using equivariance again,
$d_{g(R_i)}(f(\mu),fh(\mu))=d_{f^{-1}g(R_i)}(\mu,h(\mu))$, and since
$g$ ranges over all of $\Map(\Sigma)$ so does $f^{-1}g$. Since the subsurfaces $R_i$ are non-peripheral, all the distances exist. Thus the second supremum equals
$L_{\calR}(h)$, proving
\[
L_{\calR}(fh)\le L_{\calR}(f)+L_{\calR}(h). \qedhere
\]
\end{proof}

\subsection{Quasi-isometrically embedded abelian subgroups}\label{subsec:abelian}

We now use $L_{\calR}$ to build quasi-isometric embeddings of $\ZZ^k$ from families of disjoint
non-peripheral subsurfaces.

\begin{proposition}\label{prop:embedding}
Assume that $\Map(\Sigma)$ is CB-generated. If $\calR=\{R_1,\dots,R_k\}$ is a family
of $k$ pairwise disjoint non-peripheral compact subsurfaces of $\Sigma$, then
$\Map(\Sigma)$ contains a quasi-isometrically embedded copy of $\ZZ^k$.
\end{proposition}

\begin{proof}
Fix a marking $\mu$ on $K_0$ and consider the length function $L_{\calR}=L_{\calR}^{\mu}$
from Definition~\ref{Def:lengthfunctionR}. Let $d_A$ be the word metric coming from any
symmetric CB generating set $A$; for instance one may take the CB generating set
$\mathcal{G}$ from Theorem~\ref{thm:generating-set} (Section~\ref{generatingset}).
By Remark~\ref{rem:Rosendal-bounded}, there exists $M<\infty$ such that
$L_{\calR}(a)\le M$ for all $a\in A$. Hence for all $g\in \Map(\Sigma)$ we have
\begin{equation}\label{eq:length-to-word}
d_A(e,g)\ \ge\ \frac{1}{M}\,L_{\calR}(g).
\end{equation}

For each $i$, choose a mapping class $g_i$ supported on $R_i$ such that $g_i$
acts loxodromically on $\calC(R_i)$. Concretely: if $\zeta(R_i)\ge 2$, take $g_i$ pseudo-Anosov on $R_i$;
if $R_i$ is an annulus with core curve $\alpha_i$, take $g_i=T_{\alpha_i}$.
Since $R_i$ is non-peripheral,
\[d_{R_i}\big(\mu,g_i(\mu)\big)\ > 0. \]
Replacing $g_i$ by a power if necessary, we may assume
\begin{equation}\label{eq:linear-progress}
d_{R_i}\big(\mu,g_i^n(\mu)\big)\ \ge\ |n|
\qquad\text{for all }n\in\ZZ,
\end{equation}
using positive translation length for pseudo-Anosov elements and the translation
action of Dehn twists on annular arc graphs (see, e.g., \cite{MM2}).

Since the $R_i$ are pairwise disjoint, the mapping classes $g_i$ commute and generate
a subgroup isomorphic to $\ZZ^k$. Define
\[
\Phi\colon \ZZ^k\to \Map(\Sigma),\qquad (a_1,\dots,a_k)\mapsto g_1^{a_1}\cdots g_k^{a_k}.
\]
The map $\Phi$ is Lipschitz (with respect to the $\ell^1$ metric on $\ZZ^k$) since
$d_A(e,g_i^{n})\le |n|\,d_A(e,g_i)$.

For the lower bound, apply \eqref{eq:length-to-word} and then evaluate $L_{\calR}$
using $g=\mathrm{id}$ in the supremum:
\begin{align*}
d_A\big(\Phi(\mathbf{a}),\Phi(\mathbf{b})\big)
&=d_A\big(e,\Phi(\mathbf{a}-\mathbf{b})\big)
\ \ge\ \frac{1}{M}\,L_{\calR}\big(\Phi(\mathbf{a}-\mathbf{b})\big)\\
&\ge \frac{1}{M}\sum_{i=1}^k d_{R_i}\left(\mu,g_i^{a_i-b_i}(\mu)\right)
\ \ge\ \frac{1}{M}\sum_{i=1}^k |a_i-b_i|,
\end{align*}
where the last inequality is \eqref{eq:linear-progress}. Hence $\Phi$ is a quasi-isometric
embedding of $\ZZ^k$ into $\Map(\Sigma)$.
\end{proof}

\subsection{Infinite coarse rank}\label{subsec:infinite-rank}
We now show that, under a mild  hypothesis on the end-complexity, $\Sigma$ contains arbitrarily large families
of pairwise disjoint non-peripheral curves.

\begin{lemma}\label{lem:np-curve-K0}
Assume $\zeta(\Sigma)\ge 4$. Then there exists an essential separating curve
$\gamma\subset K_0$ such that $\gamma$ is non-peripheral. More precisely, one may choose $\gamma$
so that each component of $\Sigma-\gamma$ contains ends from at least two distinct members of
$\calA\cup\calP$ (in the sense of Section~\ref{sec:nondisplaceable}).
\end{lemma}

\begin{proof}
Since $\zeta(\Sigma)\ge 4$, the minimal anchor surface $K_0$ has at least four boundary components
(Section~\ref{sec:nondisplaceable}). As $K_0$ is finite type, there exists an essential separating
curve $\gamma\subset K_0$ so that each component of $K_0-\gamma$ contains at least two boundary
components of $K_0$. Equivalently, each component of $\Sigma-\gamma$ contains ends coming from at least
two distinct complementary components of $\Sigma-K_0$, hence from at least two distinct sets in
$\calA\cup\calP$. Denote the components of $\Sigma - \gamma$ by $\Sigma'$ and $\Sigma''$.

We claim that such a $\gamma$ is non-peripheral. Suppose for contradiction that $\gamma$ is peripheral.
Then by Lemma~\ref{lem:intersect-anchor} there exists $g\in\Map(\Sigma)$ such that
$g(\gamma)\cap K_0=\emptyset$, hence $g(\gamma)$ is contained in a single component $\Sigma_C$ of
$\Sigma-K_0$, where $C\in\calA\cup\calP$ and $\End(\Sigma_C)=C$ (notation as in
Section~\ref{sec:nondisplaceable}). That means the ends in one component of $\Sigma-\gamma$
(say $\Sigma'$) are entirely mapped into $C$. We show this contradicts the minimality of $K_0$.

We start by recalling the structure of $\calA \cup \calP$ for a minimal $K_0$.
The maximal ends $\calM(\Sigma)$ are distributed among sets in $\calA$, with each maximal end of
Cantor type appearing in exactly two such sets and each isolated maximal end appearing in
a separate set. Also, for every $P \in \calP$, there is a point $y \in P$ and $A \in \calA$
such that $\calM(A) = x_A$ is an isolated point and $x_A$ is the only maximal end that is
an accumulation point of $E(y)$ (i.e., $E(y)$ uniquely accumulates to $x_A$). In particular
$E(y) \subset A \cup P$. Also, recall that for any $z \in \End(\Sigma)$ and $g \in \Map(\Sigma)$,
we have
\[
g(E(z)) = E(z).
\]
We argue in several cases.

Assume there are $A, B \in \calA$ such that $A\sqcup B \subset \End(\Sigma')$ and hence
$g(A \sqcup B) \subset C$. From the above discussion, we see that $g(A \sqcup B) \subset C$ is possible
only if $\calM(A)$ and $\calM(B)$ are both of the same Cantor type $E(x)$ and
$E(x) = \calM(A) \cup \calM(B)$. This means $C$ must be either $A$ or $B$ (say $A$).
But then $g(\End(\Sigma''))$ contains $B$, which is a contradiction since
$E(x) \cap \End(\Sigma'') = \emptyset$ (recall that $A$ and $B$ are the only components containing $E(x)$).

Next assume there is $A \in \calA$ and $P \in \calP$ such that $A\sqcup P \subset \End(\Sigma')$.
Then $C$ has to be homeomorphic to $A$. After composing $g$ with some finite order element $h_\sigma \in \calG$,
we can assume $C=A$. Let $y \in P$ be the end such that $E(y)$ uniquely accumulates to
some maximal end $x_B$. Since $P$ can be mapped inside $A$, we must have $B=A$.
Which means
\[
E(y) \cap \End(\Sigma'') = \emptyset.
\]
This is a contradiction since $g(A \cup P ) \subset A$
implies
\[
g(\End(\Sigma'')) \supset P \supset E(y) = g(E(y)).
\]

Finally assume there are $P_1, P_2 \in \calP$ such that $P_1\sqcup P_2 \subset \End(\Sigma')$.
Then, for $i=1,2$, there is a point $y_i \in P_i$ such that $E(y_i)$
uniquely accumulates to $x_{A_i}$ for $A_i \in \calA$. By minimality of $K_0$, $y_1$ and $y_2$
are different types, otherwise, we could replace them with $P = P_1 \cup P_2$ and reduce the number of
elements in $\calP$. Therefore, there does not exist a set in $\calA \cup \calP$ that intersects both
$E(y_1)$ and $E(y_2)$. Hence, $g(P_1 \sqcup P_2) \subset C$ is not possible.

Since we arrived at a contradiction in all 3 cases, no such $g$ exists, and
$\gamma$ is non-peripheral.
\end{proof}

\begin{theorem}\label{thm:infinite-rank}
Assume $\Sigma$ is a stable infinite-type surface such that $\Map(\Sigma)$ is
CB-generated and $\zeta(\Sigma)\ge 4$.
Then for every integer $k\ge 1$ there exist $k$ pairwise disjoint non-peripheral
curves $\gamma_1,\dots,\gamma_k$ in $\Sigma$. Consequently, $\Map(\Sigma)$ has infinite coarse rank.
\end{theorem}

\begin{proof}
Let $K_0$ be the fixed minimal anchor surface from Section~\ref{sec:nondisplaceable}, and let
$\mathcal{G}$ be the CB generating set from Theorem~\ref{thm:generating-set}
(Section~\ref{generatingset}). Since $\Sigma$ has infinite type, the end space contains infinitely many ends, which implies (by the self-similarity structure of stable surfaces) that at least one of the sets $W_{A,B}$ or $W_{A}$ in the decomposition from Lemma~\ref{lem:ends} is non-empty. Hence, at least one of
the shift maps in $\mathcal{G}$ is nontrivial. Fix such a shift map $f$ (either $f=f_{A,B}$
for some $A\neq B$ or $f=f_A$ for some $A$ and $P$ as in Section~\ref{generatingset}).

By Lemma~\ref{lem:np-curve-K0}, choose a non-peripheral essential separating curve
$\gamma\subset K_0$ that separates $A$ from $B$ in the first case and separates $A$ from $P$ in
the second case. The curve $\gamma$ intersects the support of $f$ so that $\gamma$ and
$f(\gamma)$ are distinct. More precisely, there is a connected subsurface
$X\subset \Sigma$ supporting $f$ and a bi-infinite family of pairwise disjoint finite-type subsurfaces
$\{X^j\}_{j\in\ZZ}\subset X$ with $f(X^j)=X^{j+1}$; choosing $\gamma$ to separate $X^{\le 0}$ from
$X^{\ge 1}$ inside $X$, we have that $f(\gamma)$ is disjoint from $\gamma$, and hence the curves
\[
\gamma_j\ :=\ f^j(\gamma)\qquad (j\in\ZZ)
\]
are pairwise disjoint.

Finally, non-peripheralness is invariant under the action of $\Map(\Sigma)$; therefore,
the $\gamma_j$ are pairwise disjoint non-peripheral curves. Applying Proposition~\ref{prop:embedding}
to the family $\calR=\{R_1,\dots,R_k\}$ (viewing each $R_i$ as an annular subsurface with core
curve $\gamma_i$) yields a quasi-isometric embedding $\ZZ^k\hookrightarrow \Map(\Sigma)$. Since $k$
is arbitrary, the coarse rank is infinite.
\end{proof}

\begin{figure}[ht]
\setlength{\unitlength}{0.01\linewidth}
\begin{picture}(100, 34) 
\put(25,-2){
    \includegraphics[width=0.5\linewidth]{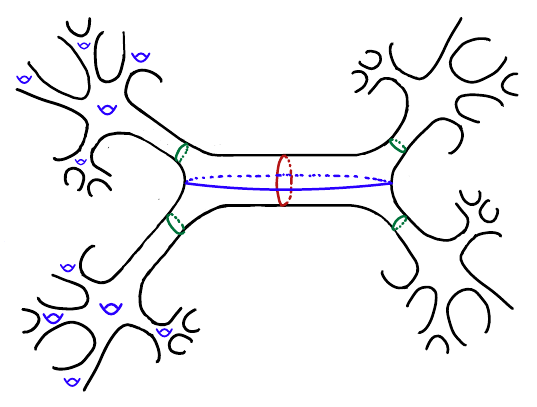}
}
\put(22,  7) {$A'$}
\put(22,  29) {$A$}
\put(77,  7) {$B'$}
\put(77,  29) {$B$}
\put(51.5,  14.5) {$\alpha$}
\put(39,  18) {$\beta$}
\end{picture}
    \caption{When $\zeta(\Sigma) = 4$ the graph $\Cnp$ may not be connected.
    But we can still find arbitrarily large number of disjoint non-peripheral curves.}
    \label{fig:4}   
\end{figure}
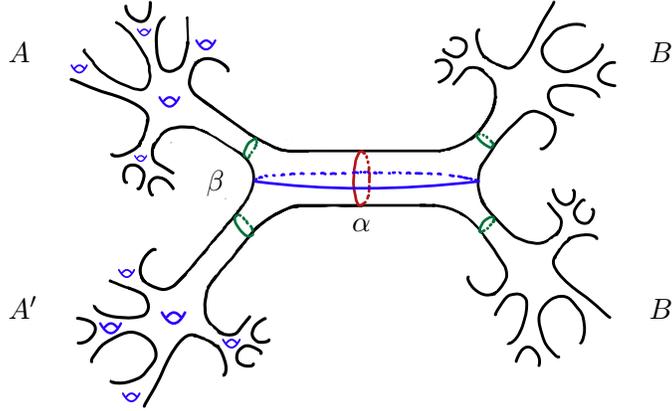

\begin{example} \label{exam:zeta-4}
The surface $\Sigma$ depicted in \figref{fig:4} has two types of maximal ends, both are cantor types. 
Hence $\zeta(\Sigma)=4$. The curve $\alpha$ is non-peripheral, but it intersects any other
non-peripheral curve (curves disjoint from $\alpha$ can always be pushed away from every compact set. 
Hence, $\Cnp$ is not connected. The curve $\beta$ is another example of a non-peripheral curve. 
There is a shift map between $A$ and $A'$ that shifts some part of the cantor set of ends
in $A$ to $A'$. The support of this shift map intersects $\beta$ and send it to a curve that is
disjoint from $\beta$. Therefore, $\Sigma$ still infinity many disjoint non-peripheral curves
and thus it has infinite coarse rank. 
\end{example}

\section{The non-peripheral curve graph}\label{sec:cnp}

In this section, we restrict our attention to non-peripheral curves and we define a 
curve graph built only from non-peripheral curves.
We will prove a basic connectivity result when $\zeta(\Sigma) \geq 5$.

\subsection{Definition and first examples}\label{subsec:cnp-definition}

Recall that a \emph{curve} on $\Sigma$ means an essential simple
closed curve, i.e.\ it does not bound a disk or once--punctured disk
(Section~\ref{sec:nondisplaceable}). A curve is \emph{non-peripheral} if its
annular neighborhood is non-peripheral in the sense of
Definition~\ref{def2}.

\begin{definition}[Non-peripheral curve graph]\label{def:Cnp}
The \emph{non-peripheral curve graph} $\Cnp(\Sigma)$ is the graph whose vertex
set is the set of isotopy classes of non-peripheral curves in $\Sigma$, and
with an edge between distinct vertices if they admit disjoint representatives.
We equip $\Cnp(\Sigma)$ with the path metric in which each edge has length $1$.
\end{definition}

When $\Sigma$ is finite-type, $\Cnp$ is equal to the curve graph. For small 
values of $\zeta(\Sigma)$ the graph $\Cnp(\Sigma)$ can be empty or
highly disconnected. For instance, if $\Sigma$ is the ladder surface 
(a two ended infinite genus surface), then $\zeta(\Sigma) =2$ and 
there are no non-peripheral curves and $\Cnp(\Sigma)=\emptyset$.
In the example of Figure \ref{fig:4}
where $\zeta(\Sigma) =4$, every curve disjoint from 
$\alpha$ is peripheral. That is, $\alpha$ intersects every other non-peripheral 
curve and hence $\Cnp$ is not connected. 

\subsection{Connectivity when $\zeta(\Sigma)\ge 5$}\label{subsec:cnp-connected}

We now prove that $\Cnp(\Sigma)$ is connected provided $\zeta(\Sigma)$ is large
enough. 

\begin{theorem}\label{thm:Cnp-connected}
Assume $\Sigma$ is stable, $\Map(\Sigma)$ is CB-generated, and $\zeta(\Sigma)\ge 5$.
Then the graph $\Cnp(\Sigma)$ is connected.
\end{theorem}

We first proceed to charaterize non-peripheral curves by the following technical lemma.
\begin{definition}\label{def:small}
We say a clopen subset $X \subset \End(\Sigma)$
is \emph{small} if there exist $A \in \calA$ and $g \in \MapSig$ such that $g(X) \subset A$.
Since every $P \in \calP$ fits inside some $A \in \calA$, if $g(X) \subset P$
then $X$ is also small.
\end{definition}

\begin{lemma}\label{lem:np-char}
Let $\Sigma$ be a connected, orientable surface of infinite type.
\begin{enumerate}
\item  If $\alpha$ is a \emph{separating} peripheral curve then 
for some component $\Sigma'$ of $\Sigma-\alpha$ the clopen set $\End(\Sigma')\subset \End(\Sigma)$ is small
(in the sense of Definition~\ref{def:small}). 
\item If $\alpha$ is \emph{non-separating}, then $\alpha$ is peripheral if and only if $\Sigma$ has infinite genus.
\end{enumerate}
\end{lemma}

\begin{proof}
If $\alpha$ is separating and peripheral then there is $g\in\Map(\Sigma)$ with 
$g(\alpha)\cap K_0=\emptyset$. Then $g(\alpha)$ lies in a single
component $\Sigma_C$ of $\Sigma-K_0$, where $C=\End(\Sigma_C)\in\calA\cup\calP$. 
Since $g(\alpha)$ is separating, one complementary component of $\Sigma-g(\alpha)$ is 
contained in $\Sigma_C$. Hence, $g(\End(\Sigma_i))\subset C$ for some $i\in\{1,2\}$,
and thus $\End(\Sigma_i)$ is small.

To see the second assertion, assume $\Sigma$ has infinite genus. 
Then there is sequence $\alpha_i$ of non-separating curves exiting a non-planar end. 
For any other non-separating $\alpha$, $\Sigma- \alpha$ is homeomorphic to $\Sigma-{\alpha_i}$ 
since the two surfaces have the same genus, end space and the same number of boundary components. 
Hence, $\alpha$ can be mapped to $\alpha_i$ and hence it can be moved away from $K_0$. 

If $\Sigma$ has finite genus, then $K_0$ carries all the genus. Any curve disjoint from $K$ lies in 
a planar neighborhood of the ends, hence is separating. Therefore, a non-separating curve can never be mapped to a curve disjoint from $K_0$. This finishes the proof. 
\end{proof}

\begin{lemma}\label{lem:small-ends}
Assume $\zeta(\Sigma) \ge 5$ and let
\[
\End(\Sigma) \ =\ X_1 \sqcup X_2 \sqcup X_3 \sqcup X_4
\]
be a decomposition of the space of ends into four clopen subsets such that
$(X_1 \sqcup X_2)$ is not small and $(X_3 \sqcup X_4)$ is not small.
Then there is $1 \le i \le 4$ such that both
\[
X_i \qquad\text{and}\qquad \End(\Sigma) - X_i
\]
are not small.
\end{lemma}

\begin{proof}
Define the end-complexity $\zeta(X)$ of a clopen set $X \subset \End(\Sigma)$ as follows.
\begin{itemize}
    \item Add $1$ for every isolated maximal end $x_A \in X$.
    \item Add $1$ if $X$ contains a point in $E(x)$ where $x$ is a Cantor-type maximal end,
    but add $2$ if $E(x)$ is entirely contained in $X$.
    \item Add $1$ for every $A \in \calA$ such that there exists $y\in \End(\Sigma)$ with
    $E(y)$ uniquely accumulating to $x_A$ and $E(y)\cap X \neq \emptyset$.
\end{itemize}
Note that if $\zeta(X) \ge 2$, then $X$ is not small. We claim that, for disjoint clopen
sets $X, X' \subset \End(\Sigma)$, we have
\[
\zeta(X) \ \le\ \zeta(X \sqcup X') \ \le\ \zeta(X) + \zeta(X').
\]
The first inequality is immediate. To see the second one, note that any contribution to
$\zeta(X \sqcup X')$ is also a contribution to $\zeta(X)$ or to $\zeta(X')$.

To finish the proof, we notice that 
\[
\sum_{i=1}^4 \zeta(X_i) \ge \zeta(\Sigma) \geq 5.
\]
Hence,
for some $i$, $\zeta(X_i) \ge 2$. Also, $(\End(\Sigma) - X_i)$ contains either
$(X_1 \sqcup X_2)$ or $(X_3 \sqcup X_4)$ and hence it is not small. This finishes
the proof.
\end{proof}

\begin{proof}[Proof of Theorem~\ref{thm:Cnp-connected}]
Choose a curve $\gamma \in \CnpS$. We show that for every $\alpha \in \CnpS$,
there is a path in $\CnpS$ obtained from surgery between
$\alpha$ and $\gamma$ connecting $\alpha$ to $\gamma$ in $\CnpS$.

The argument is similar to \cite[Section~4]{SchleimerNotes} and it proceeds by induction on 
the intersection number between $\alpha$ and $\gamma$. As a base case, assume $i(\alpha,\gamma)=1$.
Then the curve $\beta$ obtained from surgery between $\alpha$
and $\gamma$ bounds a once-punctured torus $T$ with one boundary component that contains $\gamma$.
If $\beta$ can be moved out of every compact set, so can $T$ and hence $\gamma$, which
is not possible (since $\gamma$ is non-peripheral). Hence $\beta \in \CnpS$ and therefore
$d_{\Cnp}(\alpha,\gamma)=2$.

If the intersection number between $\alpha$ and $\gamma$ is larger than $1$, we follow
$\alpha$ starting from an intersection point with $\gamma$ until it hits $\gamma$ again.
If $\alpha$ comes back to $\gamma$ on the same side of $\gamma$ as when we started, we
continue following $\alpha$ until we have one more intersection point. That is,
we find a sub-arc $\omega$ of $\alpha$ that starts on one side of $\gamma$,
ends on the other side of $\gamma$, and whose interior is either disjoint from $\gamma$ or
intersects $\gamma$ once.

If the interior of $\omega$ is disjoint from $\gamma$, then $\omega$ together with a sub-arc of
$\gamma$ forms a curve $\beta$ that intersects $\gamma$ once (after isotopy). This means
$\gamma$ and $\beta$ are both non-separating and hence $\beta$ is non-peripheral 
(by Lemma \ref{lem:np-char}, if one non-separating curve is non-peripheral, then they are all 
non-peripheral). The curve $\beta$ intersects $\alpha$ fewer times than $\gamma$ does
and $d_{\Cnp}(\gamma,\beta)=2$. Therefore, by induction, there is a path in $\CnpS$ connecting $\beta$ to 
$\alpha$ and thus to $\gamma$.

Assume the interior of $\omega$ intersects $\gamma$ once. Then $\gamma \cup \omega$
fill a four--punctured sphere $R$. We argue in two cases. If $\gamma$ is non-separating, 
then some boundary component $\beta$ of $R$ is also non-separating and hence
non-peripheral, and $\beta$ intersects $\alpha$ less than $\gamma$ does. Similarly, 
if $\alpha$ is non-separating we argue the same way, switching the roles of 
$\alpha$ and $\gamma$. 

Finally assume that the interior of $\omega$ intersects $\gamma$ once and that both 
$\alpha$ and $\gamma$ are separating. 
Suppose in the first case that $\zeta(\Sigma) \geq 5$, then $\omega \cup \gamma$ decomposes the ends of $\Sigma$ 
into four clopen sets $X_1,\dots,X_4$. We choose the indices such that $\gamma$ gives the decomposition
\[
\End(\Sigma)=(X_1 \sqcup X_2)\bigsqcup (X_3 \sqcup X_4).
\]
Since $\gamma$ is separating and non-peripheral, by Lemma~\ref{lem:np-char} the sets $(X_1 \sqcup X_2)$
and $(X_3 \sqcup X_4)$ are not small, and the hypotheses of Lemma~\ref{lem:small-ends} hold.
Now Lemma~\ref{lem:small-ends} implies that one of the curves $\beta$ obtained by surgery between
$\gamma$ and $\omega$ also decomposes $\End(\Sigma)$ into two clopen pieces that are not small,
and hence $\beta$ is non-peripheral by Lemma~\ref{lem:np-char}. Then $d_{\Cnp}(\beta,\gamma)=1$ and 
$\beta$ intersects $\alpha$ fewer times than $\gamma$ does. Again, we are done by induction.
\end{proof}

\begin{remark} \label{rem:uniform}
The proof above uses induction on the intersection number.
At each step, we find a curve $\beta$ where $d_{\Cnp}(\beta,\gamma)\leq 2$ and $\beta$ intersects $\alpha$ fewer times than $\gamma$ does.
Hence, if $\zeta(\Sigma)\geq 5$, for every $\alpha, \beta \in \Cnp$, we have 
$$
d_{\Cnp}(\alpha, \beta) \leq 2 i(\alpha, \beta) + 2.
$$
\end{remark}

\section{Hyperbolicity of the non-peripheral curve graph}\label{sec:cnp-hyp}

In this section we prove that, when $\zeta(\Sigma)\ge 5$, the non-peripheral curve graph
$\Cnp(\Sigma)$ is Gromov hyperbolic. The argument compares $\Cnp(\Sigma)$ to the grand arc
graph constructed in \cite{BNV}, and uses a general electrification criterion
(\cite[Proposition~2.6]{KR14}).

\subsection{Grand arcs and their graph}\label{subsec:grand-arcs}

Recall that $\calM(\Sigma)$ denotes the space of maximal ends of $\Sigma$, and that the
\emph{grand splitting} partitions $\calM(\Sigma)$ into \emph{maximal types}, namely the
self-similarity equivalence classes considered in \cite{BNV}. In particular, $\calM(\Sigma)$ is the disjoint union of these maximal types.

A \emph{grand arc} is a properly embedded essential arc in $\Sigma$ whose two ideal endpoints
lie in distinct maximal types, taken up to isotopy rel.\ endpoints. The \emph{grand arc
graph} $\mathcal{GA}(\Sigma)$ is the graph whose vertices are grand arcs and whose edges join
vertices that admit disjoint representatives.

\begin{lemma}\label{lem:zeta-implies-grandarc-nontrivial}
Assume $\Sigma$ is stable and $\zeta(\Sigma)\ge 5$. Then there are at least three distinct maximal types.
In particular, $\mathcal{GA}(\Sigma)$ is nonempty. Moreover, $\mathcal{GA}(\Sigma)$ is Gromov hyperbolic.
\end{lemma}

\begin{proof}
By the definition of $\zeta(\Sigma)$ (Section~\ref{sec:nondisplaceable}) and the stability hypothesis,
$\zeta(\Sigma)\ge 5$ implies that $\calM(\Sigma)$ contains at least three distinct maximal types.
Hence grand arcs exist.

Since $\Sigma$ is stable, there are only finitely many maximal types. Therefore
\cite[Theorem~1.1]{BNV} applies, and $\mathcal{GA}(\Sigma)$ is Gromov hyperbolic.
\end{proof}

\subsection{From grand arcs to non-peripheral curves}\label{subsec:arc-to-curve}

Fix $\zeta(\Sigma)\ge 5$ for the rest of this section. Let $\omega$ be a grand arc with endpoints in two
distinct maximal types. Choose disjoint stable neighborhoods $U_A,U_B$ of its two endpoints in the
sense of Section~\ref{sec:nondisplaceable}. Let $\alpha=\alpha(\omega;U_A,U_B)$ be a boundary component
of a regular neighborhood of
\[
\omega\cup \partial U_A\cup \partial U_B
\]
which separates $U_A\cup U_B$ from the complement. We call such an $\alpha$ a
\emph{small neighborhood curve} of $\omega$.

\begin{lemma}\label{lem:small-neighborhood-np}
For $U_A,U_B$ chosen sufficiently small, every small neighborhood curve
$\alpha(\omega;U_A,U_B)$ is non-peripheral.
\end{lemma}

\begin{proof}
Let $\alpha=\alpha(\omega;U_A,U_B)$ and write
\[
\Sigma\setminus \alpha=Y_1\sqcup Y_2,
\qquad
U_A\cup U_B\subset Y_1.
\]
By construction, $Y_1$ contains ends from two distinct maximal types, so for $U_A$ and $U_B$
sufficiently small the end space $\End(Y_1)$ is not small in the sense of
Definition~\ref{def:small}. Since there are at least three maximal types
(Lemma~\ref{lem:zeta-implies-grandarc-nontrivial}), the complementary region $Y_2$ contains a maximal end
of a third type, and hence $\End(Y_2)$ is also not small. Therefore neither complementary component
of $\Sigma\setminus \alpha$ has small end space, so $\alpha$ is non-peripheral.
\end{proof}

Although the curve $\alpha(\omega;U_A,U_B)$ depends on the auxiliary choices, the ambiguity is
uniformly bounded in $\Cnp(\Sigma)$.

\begin{lemma}\label{lem:small-neighborhood-diam2}
For every grand arc $\omega$, the set of all small neighborhood curves associated to $\omega$
has diameter at most $2$ in $\Cnp(\Sigma)$.
\end{lemma}

\begin{proof}
Let $\alpha=\alpha(\omega;U_A,U_B)$ and $\alpha'=\alpha(\omega;U_A',U_B')$ be two such curves.
Choose smaller stable neighborhoods
\[
U_A^\ast\subset U_A\cap U_A',
\qquad
U_B^\ast\subset U_B\cap U_B',
\]
and let $\beta=\alpha(\omega;U_A^\ast,U_B^\ast)$.
By construction, $\beta$ can be realized disjointly from both $\alpha$ and $\alpha'$.
Hence
\[
d_{\Cnp}(\alpha,\alpha')\le d_{\Cnp}(\alpha,\beta)+d_{\Cnp}(\beta,\alpha')\le 2. \qedhere
\]
\end{proof}

\subsection{A hybrid graph and quasi-isometry to $\Cnp(\Sigma)$}\label{subsec:hybrid-graph}

Let $\mathcal Y$ be the graph with vertex set
\[
V(\mathcal Y)=V(\Cnp(\Sigma))\sqcup V(\mathcal{GA}(\Sigma)),
\]
where two vertices are joined by an edge whenever the corresponding representatives can be realized
disjointly. In particular, $\Cnp(\Sigma)$ is an induced subgraph of $\mathcal Y$.

\begin{lemma}\label{lem:Cnp-qis-Y}
The inclusion $\Cnp(\Sigma)\hookrightarrow \mathcal Y$ is a quasi-isometry. More precisely:
\begin{itemize}
\item every vertex of $\mathcal Y$ lies at distance at most $1$ from $\Cnp(\Sigma)$; and
\item there exists $A\ge 1$ such that for any $\alpha,\beta\in V(\Cnp(\Sigma))$,
\[
d_{\mathcal Y}(\alpha,\beta)\le d_{\Cnp}(\alpha,\beta)\le A\,d_{\mathcal Y}(\alpha,\beta)+A.
\]
\end{itemize}
\end{lemma}

\begin{proof}
Let $\omega\in V(\mathcal{GA}(\Sigma))$. By Lemma~\ref{lem:small-neighborhood-np}, there exists a
small neighborhood curve $\alpha\in V(\Cnp(\Sigma))$ disjoint from $\omega$, so
$d_{\mathcal Y}(\omega,\alpha)=1$. Thus $\Cnp(\Sigma)$ is $1$-dense in $\mathcal Y$.

The inequality
\[
d_{\mathcal Y}(\alpha,\beta)\le d_{\Cnp}(\alpha,\beta)
\]
is immediate, since any path in $\Cnp(\Sigma)$ is also a path in $\mathcal Y$.

For the reverse inequality, let
\[
\gamma=(v_0,\dots,v_m)
\]
be a path in $\mathcal Y$ from $\alpha=v_0$ to $\beta=v_m$, where each $v_i$ is either a non-peripheral
curve or a grand arc. Replace each grand arc vertex $v_i=\omega_i$ by a choice of small neighborhood
curve $\widehat\omega_i\in V(\Cnp(\Sigma))$.

If two consecutive vertices of $\gamma$ are both grand arcs and are disjoint, then by choosing their
stable neighborhoods sufficiently small, the corresponding small neighborhood curves can be realized
disjointly as well. If one vertex is a grand arc and the other is a curve disjoint from it, then the
small neighborhood curve may be chosen disjoint from that curve. Therefore each edge of $\gamma$ is
replaced by a path in $\Cnp(\Sigma)$ of uniformly bounded length. It follows that there exists a constant
$A\ge 1$ such that $\gamma$ gives rise to a path in $\Cnp(\Sigma)$ from $\alpha$ to $\beta$ of length
at most $Am+A$, and hence
\[
d_{\Cnp}(\alpha,\beta)\le A\,d_{\mathcal Y}(\alpha,\beta)+A.
\]
\end{proof}

Consequently, in order to prove hyperbolicity of $\Cnp(\Sigma)$, it suffices to prove hyperbolicity
of $\mathcal Y$.

\subsection{Electrifying the grand arc graph}\label{subsec:electrify}

For each $\alpha\in V(\Cnp(\Sigma))$, let $\mathcal{GA}_\alpha\subset \mathcal{GA}(\Sigma)$
denote the induced subgraph spanned by all grand arcs disjoint from $\alpha$.

\begin{lemma}\label{lem:GAalpha-quasiconvex}
There exists $C\ge 0$ such that for every $\alpha\in V(\Cnp(\Sigma))$, the subgraph
$\mathcal{GA}_\alpha$ is $C$-quasiconvex in $\mathcal{GA}(\Sigma)$.
\end{lemma}

\begin{proof}
Let $\omega_1,\omega_2\in V(\mathcal{GA}_\alpha)$. By construction, both arcs are disjoint from $\alpha$.
The unicorn-path construction of \cite{BNV} produces a uniform unparameterized quasi-geodesic in
$\mathcal{GA}(\Sigma)$ joining $\omega_1$ to $\omega_2$.

Since $\alpha$ is disjoint from both endpoints, the unicorn surgeries may be performed in the
complement of $\alpha$, so every intermediate grand arc remains disjoint from $\alpha$. Thus the
entire unicorn path lies in $\mathcal{GA}_\alpha$.

By Lemma~\ref{lem:zeta-implies-grandarc-nontrivial}, $\mathcal{GA}(\Sigma)$ is hyperbolic.
Uniform unparameterized quasi-geodesics in a hyperbolic graph stay in a bounded neighborhood of a
geodesic with the same endpoints. Hence every geodesic joining $\omega_1$ to $\omega_2$ lies in a
uniform neighborhood of $\mathcal{GA}_\alpha$. This proves uniform quasiconvexity.
\end{proof}

Let $\widehat{\mathcal{GA}}(\Sigma)$ be the graph obtained from $\mathcal{GA}(\Sigma)$ by
electrifying the family $\{\mathcal{GA}_\alpha\}_{\alpha\in V(\Cnp(\Sigma))}$, namely by adding an edge
between any two vertices lying in a common $\mathcal{GA}_\alpha$.

\begin{lemma}\label{lem:GAhat-hyp}
The graph $\widehat{\mathcal{GA}}(\Sigma)$ is Gromov hyperbolic.
\end{lemma}

\begin{proof}
By Lemma~\ref{lem:zeta-implies-grandarc-nontrivial}, $\mathcal{GA}(\Sigma)$ is hyperbolic, and by
Lemma~\ref{lem:GAalpha-quasiconvex}, the family $\{\mathcal{GA}_\alpha\}$ is uniformly quasiconvex.
Therefore \cite[Proposition~2.6]{KR14} implies that $\widehat{\mathcal{GA}}(\Sigma)$ is hyperbolic.
\end{proof}

Now let $\mathcal Y_0$ be the graph whose vertex set is
\[
V(\mathcal Y_0)=V(\mathcal{GA}(\Sigma))\sqcup V(\Cnp(\Sigma)),
\]
which contains all edges of $\mathcal{GA}(\Sigma)$, and in which for each
$\alpha\in V(\Cnp(\Sigma))$ and each grand arc $\omega\in V(\mathcal{GA}_\alpha)$ we add an edge
between $\alpha$ and $\omega$.

\begin{lemma}\label{lem:Y0-qis-GAhat}
The graph $\mathcal Y_0$ is quasi-isometric to $\widehat{\mathcal{GA}}(\Sigma)$. In particular,
$\mathcal Y_0$ is Gromov hyperbolic.
\end{lemma}

\begin{proof}
If $\omega,\omega'\in \mathcal{GA}_\alpha$, then
\[
d_{\widehat{\mathcal{GA}}}(\omega,\omega')=1,
\]
while in $\mathcal Y_0$ there is a path
\[
\omega-\alpha-\omega'
\]
of length $2$. Conversely, every subpath of the form
\[
\omega-\alpha-\omega'
\]
in $\mathcal Y_0$ may be replaced by the corresponding electrified edge in
$\widehat{\mathcal{GA}}(\Sigma)$. Hence the arc-vertex subgraph of $\mathcal Y_0$ is quasi-isometric
to $\widehat{\mathcal{GA}}(\Sigma)$.

It remains to note that every curve vertex $\alpha\in V(\Cnp(\Sigma))$ lies at uniformly bounded
distance from the arc-vertex subgraph. Indeed, since $\alpha$ is compact and there are at least
three maximal types, one can choose a grand arc disjoint from $\alpha$ by taking endpoints in two
distinct maximal types lying outside a sufficiently large compact neighborhood of $\alpha$.
Thus $\mathcal{GA}_\alpha\neq\emptyset$, and every curve vertex is adjacent to some grand arc.

Therefore $\mathcal Y_0$ is quasi-isometric to $\widehat{\mathcal{GA}}(\Sigma)$, and hence
$\mathcal Y_0$ is hyperbolic.
\end{proof}

\subsection{Hyperbolicity of $\mathcal Y$ and of $\Cnp(\Sigma)$}\label{subsec:Cnp-hyp}

Recall that $\mathcal Y$ has the same vertex set as $\mathcal Y_0$, but also contains the
curve--curve disjointness edges of $\Cnp(\Sigma)$.

\begin{lemma}\label{lem:Y-qis-Y0}
The identity map on vertices induces a quasi-isometry $\mathcal Y_0\to \mathcal Y$.
In particular, $\mathcal Y$ is Gromov hyperbolic.
\end{lemma}

\begin{proof}
The identity map $\mathcal Y_0\to \mathcal Y$ is $1$-Lipschitz, since $\mathcal Y$ is obtained from
$\mathcal Y_0$ by adding edges. It therefore suffices to show that every curve--curve edge of
$\mathcal Y$ is realized by a path of uniformly bounded length in $\mathcal Y_0$.

Let $\alpha,\beta\in V(\Cnp(\Sigma))$ be disjoint. Since $\alpha\cup\beta$ is compact and there are at
least three maximal types, we may choose two maximal ends of distinct types that lie in the complement
of a sufficiently large compact neighborhood of $\alpha\cup\beta$. Joining these ends by a grand arc
$\omega$ disjoint from $\alpha\cup\beta$, we obtain
\[
\omega\in \mathcal{GA}_\alpha\cap \mathcal{GA}_\beta.
\]
Hence in $\mathcal Y_0$ there is a path
\[
\alpha-\omega-\beta
\]
of length $2$. Thus every curve--curve edge in $\mathcal Y$ has length at most $2$ in $\mathcal Y_0$,
which proves the quasi-isometry claim.

Since $\mathcal Y_0$ is hyperbolic by Lemma~\ref{lem:Y0-qis-GAhat}, so is $\mathcal Y$.
\end{proof}

We can now state the main theorem of this section.

\begin{theorem}\label{thm:Cnp-hyperbolic}
Assume $\Sigma$ is stable and $\zeta(\Sigma)\ge 5$. Then the non-peripheral curve graph
$\Cnp(\Sigma)$ is Gromov hyperbolic.
\end{theorem}

\begin{proof}
By Lemma~\ref{lem:Y-qis-Y0}, the graph $\mathcal Y$ is hyperbolic.
By Lemma~\ref{lem:Cnp-qis-Y}, $\Cnp(\Sigma)$ is quasi-isometric to $\mathcal Y$.
Since hyperbolicity is a quasi-isometry invariant, $\Cnp(\Sigma)$ is hyperbolic.
\end{proof}

\begin{remark}\label{rem:GA-vs-Cnp}
The grand arc graph and $\Cnp(\Sigma)$ need not be quasi-isometric. 
For example, suppose $|\calP|=2$ and $|\calA|\ge 3$ in the notation of Section~\ref{sec:nondisplaceable}.
Then there is a non-peripheral curve $\alpha\subset K_0$ which separates the sets in $\calP$ from the
sets in $\calA$. Let $W$ be the component of $K_0\setminus \alpha$ on the side containing the sets in
$\calA$. Then every grand arc joining two distinct maximal types must intersect $W$, so $W$ is a witness
for $\mathcal{GA}(\Sigma)$. A pseudo-Anosov supported on $W$ therefore acts loxodromically on
$\mathcal{GA}(\Sigma)$, and in particular the set of grand arcs disjoint from $\alpha$ has infinite
diameter in $\mathcal{GA}(\Sigma)$.

This is compatible with the argument above: the electrification step uses uniform quasiconvexity of the
subgraphs $\mathcal{GA}_\alpha$, not bounded diameter.
\end{remark}

\section{Divergence of big mapping class groups}\label{sec:divergence}

As before, assume that $\Sigma$ is stable and that $\MapSig$ is CB-generated. In this section we prove
a quadratic upper bound on the divergence of $\MapSig$. Divergence is a quasi-isometry invariant of
geodesic metric spaces which measures the length of the shortest detours avoiding large balls.
Although divergence is most often studied for finitely generated groups, the definition applies in the
present setting because $\Map(\Sigma)$ is CB-generated and hence carries a well-defined quasi-isometry
type of word metrics (Section~\ref{sec:nondisplaceable}).

\subsection{Divergence}\label{subsec:div-def}

We recall the definition from \cite{DMS}.  (In \cite{DMS} this is stated for proper
geodesic spaces; the definition below makes sense for any geodesic metric space, and
this is the level at which we use it.)

\begin{definition}[Divergence]\label{def:Div}
Let $(X,d_X)$ be a geodesic metric space. Fix constants $0<\rho<1$ and $T>0$.
For a pair of points $a,b\in X$ and a point $c\in X-\{a,b\}$, define the
\emph{divergence of $(a,b)$ relative to $c$} to be the length of the shortest path from
$a$ to $b$ in $X$ avoiding the open ball $B(c,r)$ of radius
\[
r=\rho\cdot \min\big(d_X(c,a),d_X(c,b)\big)-T.
\]
If no such path exists, define this divergence to be $\infty$.
The \emph{divergence of the pair $(a,b)$} is the supremum of these quantities over all
$c\in X-\{a,b\}$.  Finally, define
\[
\Div_X(R;\rho,T)=\sup\Big\{\Div_X(a,b;\rho,T)\ \Big|\ d_X(a,b)\le R\Big\},
\]
where $\Div_X(a,b;\rho,T)$ denotes the divergence of the pair $(a,b)$ (i.e.\ the supremum over $c$).
As usual, we consider divergence functions up to the equivalence relation
\[
f\preceq g\ \Longleftrightarrow\ \exists C>1\ \forall R\ge 0:\ 
f(R)\le C\,g(CR)+CR+C,
\]
and
\[
f\equiv g\ \Longleftrightarrow\ f\preceq g\quad \text{and}\quad g\preceq f.
\]
\end{definition}

\subsection{Quadratic upper bound}\label{subsec:div-main}

Our goal is to prove that if $\zeta(\Sigma)\ge 5$ then $\Map(\Sigma)$ has at most
quadratic divergence.

\begin{theorem}\label{thm:Div}
Assume $\Sigma$ is stable, $\Map(\Sigma)$ is CB-generated, and $\zeta(\Sigma)\ge 5$.
Equip $\Map(\Sigma)$ with the word metric $d_{\calS}$ associated to any CB generating set.
Then for every $0<\rho<1$ and $T>0$
\[
\Div_{(\Map(\Sigma),d_{\mathcal{S}})}(R;\rho,T) \preceq R^2. 
\]
In other words, the divergence of $\Map(\Sigma)$ is at most quadratic.
\end{theorem}

\begin{remark}\label{rem:Div-QI}
By \cite[Section~3]{DMS}, divergence is a quasi-isometry invariant for geodesic metric spaces
(note that properness is not needed). Also, for CB-generated Polish groups the quasi-isometry type
of a word metric does not depend on the particular symmetric CB generating set
(Section~\ref{sec:nondisplaceable}, \cite{Rosendal}).
Thus it is enough to work with $d_{\mathcal{S}}$ for the specific generating set
$\mathcal{S}$ from Theorem~\ref{thm:generating-set}.
Furthermore, the equivalence class of divergence does not depend on $\rho$ and $T$
(\cite[Section~3]{DMS}). Hence we can use any pair $(\rho,T)$.
\end{remark}

The left action of the group on itself is by isometries of $d_{\mathcal{S}}$, so any triple
$(a,b,c)$ can be translated so that $c=\id$. Since we are taking the supremum over all triples
$(a,b,c)$, it suffices to take the supremum over triples with $c=\id$. Also, following
\cite[Section~3]{DMS}, we may assume that $d_{\mathcal{S}}(\id,a)$,
$d_{\mathcal{S}}(\id,b)$, and $d_{\mathcal{S}}(a,b)$ are all comparable to $R$ up to a fixed
multiplicative constant. Therefore, Theorem~\ref{thm:Div} follows from Theorem~\ref{thm:Div2} below.

\begin{theorem}\label{thm:Div2}
Assume $\Sigma$ is stable, $\Map(\Sigma)$ is CB-generated, and $\zeta(\Sigma)\ge 5$.
Then there exist constants $\eta>0$, $M_0>0$, and $T_0>0$, depending only on $\Sigma$
and the chosen CB generating set $\mathcal{S}$, such that the following holds for all $R>0$.

If $g_1,g_2\in \Map(\Sigma)$ satisfy
\[
R \le \Norm{g_i}_{\mathcal{S}} \le 2R \qquad (i=1,2),
\]
then there exists a path in $(\Map(\Sigma),d_{\mathcal{S}})$ from $g_1$ to $g_2$ which is disjoint
from the ball
\[
B(\id,\eta R-T_0)
\]
and has length at most $M_0R^2$.
\end{theorem}

\subsection{A twist length function}\label{subsec:div-step1}

We use the annular specialization of the length functions from
Section~\ref{subsec:LcalR} (Definition~\ref{Def:lengthfunctionR}). Fix once and for all a marking $\mu$ on $K_0$ (Section~\ref{sec:nondisplaceable}).
For a non-peripheral curve $\alpha$, let $A_\alpha$ denote the annulus with core $\alpha$,
and define
\begin{equation}\label{eq:Lalpha-def}
L_\alpha(g):=\sup_{h\in \Map(\Sigma)} \tw_{h(\alpha)}\big(\mu,g(\mu)\big),
\end{equation}
where $\tw_{h(\alpha)}$ is the annular distance from Definition~\ref{def:Twisting}
(Section~\ref{lf}).  This is the special case of $L_{\calR}$ with $\calR=\{A_\alpha\}$.

\begin{lemma}\label{lem:Lalpha-length}
For every non-peripheral curve $\alpha$, the function $L_\alpha$ is a length function on $\Map(\Sigma)$.
Moreover, there exists $M_\alpha<\infty$ such that
\begin{equation}\label{eq:Lalpha-wordlower}
\Norm{g}_{\mathcal{S}}\ge \frac{1}{M_\alpha}\,L_\alpha(g)\qquad\text{for all }g\in\Map(\Sigma).
\end{equation}
\end{lemma}

\begin{proof}
The fact that $L_\alpha$ is a length function is exactly Lemma~\ref{Lem:islengthfunction}
applied to $\calR=\{A_\alpha\}$ (Section~\ref{subsec:LcalR}).  Since $\mathcal{S}$ is CB, by
\cite[Proposition~2.7(5)]{Rosendal} and Remark~\ref{rem:Rosendal-bounded},
$L_\alpha$ is bounded on $\mathcal{S}$: set
\[
M_\alpha:=\sup_{s\in\mathcal{S}} L_\alpha(s)<\infty.
\]
Then \eqref{eq:Lalpha-wordlower} follows from subadditivity:
if $g=s_1\cdots s_n$ with $s_i\in\mathcal{S}$, then
\[
L_\alpha(g)\le \sum_{i=1}^n L_\alpha(s_i)\le n M_\alpha,
\]
hence
$\Norm{g}_{\mathcal{S}}=n\ge L_\alpha(g)/M_\alpha$.
\end{proof}

\begin{lemma}\label{lem:twist-linear}
Let $\alpha$ be a non-peripheral curve and let $D_\alpha$ be the Dehn twist about $\alpha$.
There exists $c_\alpha>0$ such that for all $n\in\ZZ$,
\[
L_\alpha(D_\alpha^n)\ge c_\alpha\,|n|.
\]
\end{lemma}

\begin{proof}
Since $\alpha$ is non-peripheral, every translate $h(\alpha)$ intersects $K_0$ essentially
(Definition~\ref{def2}), hence $\pi_{h(\alpha)}(\mu)\neq\emptyset$ and the twisting
distance $\tw_{h(\alpha)}(\mu,D_\alpha^n(\mu))$ is defined.
Taking $h=\id$ in \eqref{eq:Lalpha-def} gives
\[
L_\alpha(D_\alpha^n)\ge \tw_{\alpha}\big(\mu,D_\alpha^n(\mu)\big),
\]
and by Remark~\ref{rem:twist-translation} the right-hand side grows coarsely linearly in $|n|$.
\end{proof}

\subsection{A projection to $\Cnp(\Sigma)$}\label{subsec:div-step2}

Assume throughout this subsection that $\zeta(\Sigma)\ge 5$ so that $\Cnp(\Sigma)$ is connected
(Theorem~\ref{thm:Cnp-connected} in Section~\ref{subsec:cnp-connected}).
Fix a base vertex $\alpha_0\in \Cnp(\Sigma)$ and define
\begin{equation}\label{eq:pi-def}
\pi:\Map(\Sigma)\to \Cnp(\Sigma),\qquad \pi(g):=g(\alpha_0).
\end{equation}

\begin{remark}\label{rem:pi-coarse-choice}
If $\alpha_0'$ is another base curve in $\Cnp(\Sigma)$, then for every $g$,
\[
d_{\Cnp}\big(g(\alpha_0),g(\alpha_0')\big)=d_{\Cnp}(\alpha_0,\alpha_0')
\]
because $\Map(\Sigma)$ acts by graph automorphisms on $\Cnp(\Sigma)$. Thus changing the base curve
changes $\pi$ by a uniformly bounded amount. For the arguments below, however, we will choose
$\alpha_0$ so that $\alpha_0\subset K_0$.
\end{remark}

\subsection{$\pi$ is Lipschitz}\label{subsec:div-step3}

\begin{lemma}\label{lem:pi-Lipschitz}
There exists $L\ge 1$ such that for every generator $s\in\mathcal{S}$,
\[
d_{\Cnp}\big(\alpha_0,s(\alpha_0)\big)\le L.
\]
Consequently, for all $g,h\in\Map(\Sigma)$,
\[
d_{\Cnp}\big(\pi(g),\pi(h)\big)\le L\,d_{\mathcal{S}}(g,h).
\]
\end{lemma}

\begin{proof}
The second statement follows from the first by writing $h=g s_1\cdots s_n$ with $n=d_{\mathcal{S}}(g,h)$
and using the triangle inequality in $\Cnp(\Sigma)$:
\[
d_{\Cnp}\big(g(\alpha_0),h(\alpha_0)\big)\le \sum_{i=1}^n d_{\Cnp}\big(gs_1\cdots s_{i-1}(\alpha_0),gs_1\cdots s_i(\alpha_0)\big)
=\sum_{i=1}^n d_{\Cnp}\big(\alpha_0,s_i(\alpha_0)\big)\le nL.
\]
Thus it remains to prove the existence of $L$.

Fix $\alpha_0\subset K_0$ once and for all.  For each of the finitely many types of generators in
Theorem~\ref{thm:generating-set} (Section~\ref{generatingset}) one checks that the geometric
intersection number $i(\alpha_0,s(\alpha_0))$ is uniformly bounded independent of $s\in\mathcal{S}$:
\begin{itemize}
\item elements of $\calV_{K_0}$ fix $K_0$ pointwise, hence fix $\alpha_0$;
\item local generators in $\mathcal{L}$ are either Dehn twists or half twists, so they change
$\alpha_0$ by a uniformly bounded amount; and
\item the shift maps and the finitely many maximal-end permutations have supports and boundary data
contained in a fixed finite-type subsurface (as in the definition of $\mathcal{L}$), so their effect on
$\alpha_0$ is controlled as well.
\end{itemize}
Given a uniform bound on $i(\alpha_0,s(\alpha_0))$ and
in view of Remark~\ref{rem:uniform}, the surgery argument from
Section~\ref{subsec:cnp-connected} produces a uniformly bounded length 
path in $\Cnp(\Sigma)$ from $\alpha_0$ to $s(\alpha_0)$.
This yields the desired $L$.
\end{proof}

\subsection{Commuting Dehn twists for generators}\label{subsec:div-step4}

\begin{lemma}\label{lem:commuting-twist}
Fix $\alpha_0\subset K_0$ once and for all. There exists a constant $C_1\ge 1$ with the following
property. For every $s\in\mathcal{S}$ there exists a curve $\alpha_s\in \Cnp(\Sigma)$ such that
\begin{itemize}
\item $d_{\Cnp}(\alpha_0,\alpha_s)\le C_1$; and
\item the Dehn twist $D_{\alpha_s}$ commutes with $s$.
\end{itemize}
\end{lemma}

\begin{proof}
We choose $\alpha_s$ by a case-by-case inspection of the generators in
Theorem~\ref{thm:generating-set} (Section~\ref{generatingset}).

If $s\in \calV_{K_0}$, take $\alpha_s=\alpha_0$, so $s$ fixes $\alpha_0$ and hence commutes with $D_{\alpha_0}$.
If $s\in\mathcal{L}$ is a Dehn twist or a half-twist around a curve $\gamma$, take $\alpha_s$ to be a non-peripheral curve disjoint from $\gamma$. Such a curve exists in the five-holed sphere guaranteed by $\zeta(\Sigma)\ge 5$ (Section~\ref{subsec:cnp-connected}).
For the finitely many maximal-end permutations and shift maps in $\mathcal{S}$, choose $\alpha_s$ among finitely many curves in the fixed finite-type region where these maps meet $K_0$, so that $\alpha_s$ is preserved by $s$.
In all cases, $s(\alpha_s)=\alpha_s$ and hence $s$ commutes with $D_{\alpha_s}$.

Since there are only finitely many types of generators, we may take $C_1$ to be the maximum of the (finite) distances $d_{\Cnp}(\alpha_0,\alpha_s)$ arising from the
choices above.
\end{proof}

\subsection{A ``linked'' word decomposition}\label{subsec:div-step5}

\begin{proposition}\label{prop:link}
There exists a constant $C_0\ge 1$ with the following property.
For every $g_1,g_2\in\Map(\Sigma)$ there exist an integer $n\ge 1$, a sequence of curves
\[
\alpha_1,\alpha_2,\dots,\alpha_n\in \Cnp(\Sigma),
\]
and a sequence of elements $s_1,\dots,s_n\in \mathcal{S}\cup\{\id\}$ such that:
\begin{itemize}
\item $D_{\alpha_i}$ commutes with $s_i$ for every $i$;
\item $d_{\Cnp}(\alpha_i,\alpha_{i+1})\le 1$ for every $1\le i<n$;
\item $s_1s_2\cdots s_n=g_1^{-1}g_2$; and
\item $n\le C_0\,\Norm{g_1^{-1}g_2}_{\mathcal{S}}$.
\end{itemize}
Moreover, the curves $\alpha_i$ can be chosen from a finite subset $\mathcal{F}\subset\Cnp(\Sigma)$
depending only on $\Sigma$ and $\mathcal{S}$.
\end{proposition}

\begin{proof}
Write $g_1^{-1}g_2$ as a word of minimal length in $\mathcal{S}$:
\[
g_1^{-1}g_2=t_1t_2\cdots t_m,\qquad t_j\in\mathcal{S},\quad m=\Norm{g_1^{-1}g_2}_{\mathcal{S}}.
\]
For each $t_j$, choose a curve $\beta_j\in\Cnp(\Sigma)$ as in Lemma~\ref{lem:commuting-twist}, so that
$D_{\beta_j}$ commutes with $t_j$.

By Lemma~\ref{lem:commuting-twist}, the curves $\beta_j$ may be chosen from a finite set
$\mathcal{B}\subset\Cnp(\Sigma)$ depending only on $\Sigma$ and $\mathcal{S}$.

For each ordered pair $(\beta,\beta')\in\mathcal{B}\times\mathcal{B}$, fix once and for all an edge path
\[
P(\beta,\beta')=(\gamma^{\beta,\beta'}_0,\gamma^{\beta,\beta'}_1,\dots,\gamma^{\beta,\beta'}_{\ell(\beta,\beta')})
\]
in $\Cnp(\Sigma)$ from $\beta$ to $\beta'$, so that
\[
\gamma^{\beta,\beta'}_0=\beta,\qquad \gamma^{\beta,\beta'}_{\ell(\beta,\beta')}=\beta',
\]
and consecutive vertices are adjacent in $\Cnp(\Sigma)$.
Let $\mathcal{F}$ be the union of the vertex sets of all these paths. Since $\mathcal{B}$ is finite,
the set $\mathcal{F}$ is finite.

We now define the sequence $(\alpha_i,s_i)$ by concatenating blocks. For each $1\le j\le m-1$, take:
\begin{itemize}
\item first the letter $t_j$ labelled by the curve $\beta_j$;
\item then, for each vertex $\gamma^{\beta_j,\beta_{j+1}}_k$ with $1\le k\le \ell(\beta_j,\beta_{j+1})-1$,
insert a letter $\id$ labelled by $\gamma^{\beta_j,\beta_{j+1}}_k$.
\end{itemize}
Finally, append the letter $t_m$ labelled by $\beta_m$.

By construction, every $s_i$ belongs to $\mathcal{S}\cup\{\id\}$, and each $D_{\alpha_i}$ commutes with $s_i$:
this is immediate for the identity letters, and for the letters $t_j$ it follows from the choice of $\beta_j$.
Also, the product of all letters is
\[
t_1t_2\cdots t_m=g_1^{-1}g_2,
\]
since the inserted identity letters do not change the product.

Moreover, consecutive curves are either equal or adjacent in $\Cnp(\Sigma)$ by construction, so in particular
\[
d_{\Cnp}(\alpha_i,\alpha_{i+1})\le 1
\qquad\text{for all }1\le i<n.
\]
Since $\mathcal{B}$ is finite, there is a uniform bound
\[
L_0:=\max\{\ell(\beta,\beta'):\beta,\beta'\in\mathcal{B}\}<\infty.
\]
Thus each block contributes at most $L_0+1$ letters, and hence
\[
n\le m(L_0+1).
\]
Therefore the conclusion holds with $C_0=L_0+1$. Finally, every curve label $\alpha_i$ lies in $\mathcal{F}$.
\end{proof}

\begin{remark}\label{rem:uniform-constants}
In Proposition~\ref{prop:link}, all twist curves belong to the fixed finite set $\mathcal{F}$.
Hence all constants associated to these curves may be chosen uniformly.

For each $\alpha\in\mathcal{F}$, let $M_\alpha$ be the constant from
Lemma~\ref{lem:Lalpha-length}, and let $(q_\alpha,Q_\alpha)$ be quasi-geodesic constants
for the twist line
\[
\ZZ\to (\Map(\Sigma),d_{\mathcal S}),\qquad m\mapsto gD_\alpha^m
\]
from Lemma~\ref{lem:twist-quasiline}.  Also choose $B_\alpha\ge 0$ so that:
\begin{enumerate}
\item the projection $\pi_\alpha(\mu)$ has diameter at most $B_\alpha$ in the annular curve graph of $\alpha$; and
\item for every $m\in\ZZ$,
\[
\tw_\alpha\big(\mu,D_\alpha^m(\mu)\big)\ge |m|-B_\alpha.
\]
\end{enumerate}
Define
\[
M:=\max_{\alpha\in\mathcal F} M_\alpha,\qquad
q:=\max_{\alpha\in\mathcal F} q_\alpha,\qquad
Q:=\max_{\alpha\in\mathcal F} Q_\alpha,
\]
and
\[
B:=\max_{\alpha\in\mathcal F} B_\alpha,\qquad
A:=2B,\qquad
D_0:=\max_{\alpha\in\mathcal F}\Norm{D_\alpha}_{\mathcal S}.
\]

Finally, choose once and for all a constant
\[
0<\eta<\frac{1}{1+q^2}.
\]
All subsequent estimates depend only on the uniform constants $M,q,Q,A,D_0,\eta$.
\end{remark}

\subsection{Detours using commuting Dehn twists}\label{subsec:div-step6}

We now complete the proof of Theorem~\ref{thm:Div2} by the standard detour construction
(compare \cite[Section~4]{DR}).

\begin{lemma}\label{lem:twist-quasiline}
Let $\alpha\in\Cnp(\Sigma)$ and let $D_\alpha$ be the Dehn twist about $\alpha$. Then the map
\[
\ZZ\to (\Map(\Sigma),d_{\mathcal S}),\qquad m\mapsto gD_\alpha^m
\]
is a quasi-isometric embedding, with constants depending only on $\alpha$ and $\mathcal S$.
\end{lemma}

\begin{proof}
Since left multiplication by $g$ is an isometry of $(\Map(\Sigma),d_{\mathcal S})$,
it suffices to consider the map $m\mapsto D_\alpha^m$.

For the upper bound,
\[
d_{\mathcal S}(D_\alpha^m,D_\alpha^n)
=\Norm{D_\alpha^{m-n}}_{\mathcal S}
\le |m-n|\,\Norm{D_\alpha}_{\mathcal S}.
\]

For the lower bound, Lemma~\ref{lem:Lalpha-length} and Lemma~\ref{lem:twist-linear} give
\[
d_{\mathcal S}(D_\alpha^m,D_\alpha^n)
=\Norm{D_\alpha^{m-n}}_{\mathcal S}
\ge \frac{1}{M_\alpha}\,L_\alpha(D_\alpha^{m-n})
\succ |m-n|.
\]
Thus $m\mapsto D_\alpha^m$ is a quasi-isometric embedding, and so is
$m\mapsto gD_\alpha^m$.
\end{proof}

\begin{lemma}[Uniform escape estimates]\label{lem:extending-branch}
Let $\mathcal F\subset\Cnp(\Sigma)$ and the constants $M,q,Q,A,\eta$ be as in
Remark~\ref{rem:uniform-constants}. Then there exists $T_1\ge 0$, depending only on
$\Sigma$ and $\mathcal S$, such that the following hold.

\begin{enumerate}
\item[\textup{(1)}] \emph{(Extending branch)} For every $\alpha\in\mathcal F$ and every
$g\in\Map(\Sigma)$, there exists a sign $\varepsilon\in\{\pm1\}$ such that
\begin{equation}\label{eq:extending-branch}
\Norm{gD_\alpha^{\varepsilon m}}_{\mathcal S}\ge \eta\,\Norm{g}_{\mathcal S}-T_1
\qquad\text{for all }m\ge 0.
\end{equation}

\item[\textup{(2)}] \emph{(Uniform linear lower bound: one twist)} For every
$\alpha\in\mathcal F$, every $g\in\Map(\Sigma)$, and every $m\in\ZZ$,
\begin{equation}\label{eq:one-twist-linear}
\Norm{gD_\alpha^m}_{\mathcal S}\ge
\frac{|m|-M\,\Norm{g}_{\mathcal S}-A}{M}.
\end{equation}

\item[\textup{(3)}] \emph{(Uniform linear lower bound: two disjoint twists)} If
$\alpha,\beta\in\mathcal F$ are disjoint, then for every $g\in\Map(\Sigma)$ and every
$m,n\in\ZZ$,
\begin{equation}\label{eq:two-twist-linear}
\Norm{gD_\alpha^mD_\beta^n}_{\mathcal S}\ge
\frac{\max\{|m|,|n|\}-M\,\Norm{g}_{\mathcal S}-A}{M}.
\end{equation}
\end{enumerate}
\end{lemma}

\begin{proof}
We begin with \textup{(1)}. Fix $\alpha\in\mathcal F$ and $g\in\Map(\Sigma)$, and set
\[
R:=\Norm{g}_{\mathcal S}.
\]
Consider the bi-infinite quasi-geodesic
\[
\gamma:\ZZ\to \Map(\Sigma),\qquad \gamma(m):=gD_\alpha^m,
\]
whose quasi-geodesic constants may be taken to be $(q,Q)$ by
Remark~\ref{rem:uniform-constants}.

Suppose, for contradiction, that both rays $\{\gamma(m)\}_{m\ge 0}$ and
$\{\gamma(-m)\}_{m\ge 0}$ meet the ball $B(\id,\eta R)$. Then there exist integers
$m_+,m_-\ge 0$ such that
\[
\Norm{\gamma(m_+)}_{\mathcal S}\le \eta R
\qquad\text{and}\qquad
\Norm{\gamma(-m_-)}_{\mathcal S}\le \eta R.
\]
By the triangle inequality,
\[
d_{\mathcal S}\bigl(\gamma(0),\gamma(m_+)\bigr)\ge (1-\eta)R,
\qquad
d_{\mathcal S}\bigl(\gamma(0),\gamma(-m_-)\bigr)\ge (1-\eta)R.
\]
Since $\gamma$ is a $(q,Q)$-quasi-geodesic, this implies
\[
m_+\ge \frac{(1-\eta)R-Q}{q},
\qquad
m_-\ge \frac{(1-\eta)R-Q}{q}.
\]
On the other hand,
\[
d_{\mathcal S}\bigl(\gamma(-m_-),\gamma(m_+)\bigr)\le 2\eta R,
\]
while the quasi-geodesic lower bound gives
\[
d_{\mathcal S}\bigl(\gamma(-m_-),\gamma(m_+)\bigr)
\ge \frac{m_++m_-}{q}-Q.
\]
Combining these inequalities yields
\[
\frac{2(1-\eta)R-2Q}{q^2}-Q\le 2\eta R.
\]
Equivalently,
\[
2\bigl(1-(1+q^2)\eta\bigr)R\le (2+q^2)Q.
\]
Since $\eta<1/(1+q^2)$, this is impossible for all sufficiently large $R$.
Hence there exists $R_0\ge 0$, depending only on $q,Q,\eta$, such that whenever
$\Norm{g}_{\mathcal S}\ge R_0$, at least one of the two rays
$\{gD_\alpha^m\}_{m\ge 0}$ or $\{gD_\alpha^{-m}\}_{m\ge 0}$ is disjoint from
$B(\id,\eta R)$.

Now choose
\[
T_1:=\eta R_0.
\]
If $R\ge R_0$, then the good ray is disjoint from $B(\id,\eta R)$ and hence certainly
from $B(\id,\eta R-T_1)$. If $R\le R_0$, then $\eta R-T_1\le 0$, so
\eqref{eq:extending-branch} is automatic. This proves \textup{(1)}.

We now prove \textup{(2)}. Fix $\alpha\in\mathcal F$, $g\in\Map(\Sigma)$, and $m\in\ZZ$.
Set
\[
\beta:=g(\alpha).
\]
Since $gD_\alpha^m=D_\beta^m g$, taking $h=g$ in the definition of $L_\alpha$ gives
\[
L_\alpha(gD_\alpha^m)\ge \tw_\beta\bigl(\mu,D_\beta^m g(\mu)\bigr).
\]
By choice of $B$, the projection $\pi_\beta(g(\mu))$ has diameter at most $B$, and
$D_\beta^m$ acts on the annular curve graph of $\beta$ by translation. Therefore
\[
\tw_\beta\bigl(g(\mu),D_\beta^m g(\mu)\bigr)\ge |m|-B.
\]
By the triangle inequality in the annular curve graph,
\[
\tw_\beta\bigl(\mu,D_\beta^m g(\mu)\bigr)
\ge |m|-B-\tw_\beta\bigl(\mu,g(\mu)\bigr).
\]
Again taking $h=g$ in \eqref{eq:Lalpha-def}, we have
\[
\tw_\beta\bigl(\mu,g(\mu)\bigr)\le L_\alpha(g).
\]
Using Lemma~\ref{lem:Lalpha-length},
\[
L_\alpha(g)\le M_\alpha \Norm{g}_{\mathcal S}\le M\,\Norm{g}_{\mathcal S}.
\]
Hence
\[
L_\alpha(gD_\alpha^m)\ge |m|-M\,\Norm{g}_{\mathcal S}-B.
\]
Applying Lemma~\ref{lem:Lalpha-length} once more, and using $A=2B\ge B$, we obtain
\[
\Norm{gD_\alpha^m}_{\mathcal S}
\ge \frac{1}{M}\,L_\alpha(gD_\alpha^m)
\ge \frac{|m|-M\,\Norm{g}_{\mathcal S}-A}{M},
\]
which is \eqref{eq:one-twist-linear}.

For \textup{(3)}, let $\alpha,\beta\in\mathcal F$ be disjoint, and fix
$g\in\Map(\Sigma)$ and $m,n\in\ZZ$. Set
\[
\delta:=g(\beta).
\]
Then
\[
gD_\alpha^mD_\beta^n=D_{g(\alpha)}^m D_\delta^n g,
\]
and the twists $D_{g(\alpha)}^m$ and $D_\delta^n$ commute, since $g(\alpha)$ and $\delta$
are disjoint. Taking $h=g$ in the definition of $L_\beta$ gives
\[
L_\beta(gD_\alpha^mD_\beta^n)
\ge \tw_\delta\bigl(\mu,D_{g(\alpha)}^m D_\delta^n g(\mu)\bigr).
\]
Because $g(\alpha)$ is disjoint from $\delta$, the twist $D_{g(\alpha)}^m$ acts trivially on
the annular curve graph of $\delta$ up to a bounded error. Since $\mathcal F$ is finite,
this bounded error may be absorbed into $B$, and hence
\[
\tw_\delta\bigl(\mu,D_{g(\alpha)}^m D_\delta^n g(\mu)\bigr)
\ge \tw_\delta\bigl(\mu,D_\delta^n g(\mu)\bigr)-B.
\]
Applying the estimate from \textup{(2)} to the pair $(\beta,g)$ yields
\[
\tw_\delta\bigl(\mu,D_\delta^n g(\mu)\bigr)
\ge |n|-M\,\Norm{g}_{\mathcal S}-B.
\]
Therefore
\[
L_\beta(gD_\alpha^mD_\beta^n)\ge |n|-M\,\Norm{g}_{\mathcal S}-2B
=|n|-M\,\Norm{g}_{\mathcal S}-A.
\]
Using Lemma~\ref{lem:Lalpha-length}, we conclude that
\[
\Norm{gD_\alpha^mD_\beta^n}_{\mathcal S}
\ge \frac{|n|-M\,\Norm{g}_{\mathcal S}-A}{M}.
\]
Interchanging the roles of $\alpha$ and $\beta$ gives the same estimate with $|m|$ in place
of $|n|$. Combining the two bounds yields \eqref{eq:two-twist-linear}.
\end{proof}

\begin{proof}[Proof of Theorem~\ref{thm:Div2}]
Fix $R>0$ and let $g_1,g_2\in\Map(\Sigma)$ satisfy
\[
R\le \Norm{g_i}_{\mathcal S}\le 2R.
\]
Set
\[
h:=g_1^{-1}g_2,
\]
so that
\[
\Norm{h}_{\mathcal S}\le 4R.
\]

Apply Proposition~\ref{prop:link} to obtain curves
\[
\alpha_1,\dots,\alpha_n\in\Cnp(\Sigma)
\qquad\text{and elements}\qquad 
s_1,\dots,s_n\in \mathcal S\cup\{\id\}
\]
such that
\[
s_1\cdots s_n=h,\qquad
D_{\alpha_i}\text{ commutes with }s_i,\qquad
d_{\Cnp}(\alpha_i,\alpha_{i+1})\le 1,
\]
and
\[
n\le C_0\,\Norm{h}_{\mathcal S}\le 4C_0R.
\]

Compress consecutive equal curves into maximal blocks. Thus we obtain curves
\[
\beta_1,\dots,\beta_r\in\mathcal F
\qquad\text{and words}\qquad  
h_1,\dots,h_r
\]
such that:
\begin{itemize}
\item $h=h_1\cdots h_r$;
\item each $h_j$ is a product of a consecutive subword of the letters $s_i$;
\item $D_{\beta_j}$ commutes with every letter in the block $h_j$, and hence with $h_j$;
\item consecutive curves $\beta_j,\beta_{j+1}$ are distinct and adjacent in $\Cnp(\Sigma)$, hence disjoint;
\item the total number of generator letters appearing in the words $h_j$ is still $n$.
\end{itemize}

For $0\le j\le r$, define
\[
u_j:=h_1\cdots h_j,\qquad p_j:=g_1u_j.
\]
Thus
\[
u_0=\id,\quad u_r=h,\quad p_0=g_1,\quad p_r=g_2.
\]
Also, for every $j$,
\[
\Norm{u_j}_{\mathcal S}\le \Norm{h}_{\mathcal S}\le 4R,
\]
and hence
\[
\Norm{p_j}_{\mathcal S}\le \Norm{g_1}_{\mathcal S}+\Norm{u_j}_{\mathcal S}\le 6R.
\]
The same estimate holds for every intermediate prefix arising while reading a block $h_j$.

Choose signs $\varepsilon_1,\varepsilon_r\in\{\pm1\}$ using
Lemma~\ref{lem:extending-branch}\textup{(1)} applied to the pairs $(\beta_1,g_1)$ and
$(\beta_r,g_2)$. For the intermediate blocks $2\le j\le r-1$, choose signs
$\varepsilon_j\in\{\pm1\}$ arbitrarily.

Let $K_{\mathrm{tw}}\ge 1$ be a constant to be chosen below, and define
\[
N:=\lceil K_{\mathrm{tw}}R\rceil.
\]
Set
\[
v:=g_1D_{\beta_1}^{\varepsilon_1 N},
\qquad
w:=g_2D_{\beta_r}^{\varepsilon_r N}.
\]
By Lemma~\ref{lem:extending-branch}\textup{(1)}, the twist-ray path from $g_1$ to $v$
is disjoint from
\[
B(\id,\eta R-T_1),
\]
and similarly the twist-ray path from $g_2$ to $w$ is disjoint from the same ball.
It therefore remains to connect $v$ to $w$ outside this ball.

\begin{figure}
\begin{tikzpicture}[line cap=round,line join=round]

\def\slabW{2.2}     
\def\slabH{1.2}     
\def\depthx{0.95}   
\def\depthy{0.55}   
\def\gap{3.1}       

\newcommand{\slabunit}[2]{%
  \coordinate (A) at (#1,#2);
  \coordinate (B) at ($(A)+(\slabW,0)$);
  \coordinate (C) at ($(B)+(\depthx,\depthy)$);
  \coordinate (D) at ($(A)+(\depthx,\depthy)$);

  \draw[dashed, draw=black, fill=red, fill opacity=0.5] (A)--(B)--(C)--(D)--cycle;

}

\newcommand{\slabunitt}[2]{%
  \coordinate (A) at (#1,#2);
  \coordinate (B) at ($(A)+(\slabW,0)$);
  \coordinate (C) at ($(B)+(\depthx,\depthy)$);
  \coordinate (D) at ($(A)+(\depthx,\depthy)$);


  \coordinate (E) at ($(D)+(0,\slabH)$);
  \coordinate (F) at ($(C)+(0,\slabH)$);
}

\slabunit{-3.6}{3.3}

\begin{scope}[thick]
  \coordinate (P0) at (-3.6,3.1);
  \coordinate (P1) at ($(P0)+(\depthx,\depthy)$);
  \coordinate (P2) at ($(P1)+(0,2)$);
  \coordinate (P3) at ($(P0)+(0,2)$);
  \draw (P0)--(P3)--(P2)--(P1)--cycle;
\end{scope}

\slabunit{0.8}{4.5}

\begin{scope}[thick]
  \coordinate (Q0) at (0.8,3.1);
  \coordinate (Q1) at ($(Q0)+(\slabW,0)$);
  \coordinate (Q2) at ($(Q1)+(\depthx,\depthy)$);
  \coordinate (Q3) at ($(Q2)+(0,2.0)$);
  \coordinate (Q4) at ($(Q1)+(0,2.0)$);
  \draw (Q1)--(Q4)--(Q3)--(Q2)--cycle;
\end{scope}

\begin{scope}[thick]
  \coordinate (Q0) at (-1.4,3.1);
  \coordinate (Q1) at ($(Q0)+(\slabW,0)$);
  \coordinate (Q2) at ($(Q1)+(\depthx,\depthy)$);
 \coordinate (Q3) at ($(Q2)+(0,2.0)$);
  \coordinate (Q4) at ($(Q1)+(0,2.0)$);

\draw (Q1)--(Q4)--(Q3)--(Q2)--cycle;
\end{scope}

\draw[dashed,  draw=black, fill=yellow, fill opacity=0.5] (-0.8, 3.4)--(1.3, 3.4)--(1.3, 5.4)--(-0.8, 5.4)--(-0.8, 3.4);

\begin{scope}[thick]
  \coordinate (Q0) at (-3.6,3.1);
  \coordinate (Q1) at ($(Q0)+(\slabW,0)$);
  \coordinate (Q2) at ($(Q1)+(\depthx,\depthy)$);
  \coordinate (Q3) at ($(Q2)+(0,2.0)$);
  \coordinate (Q4) at ($(Q1)+(0,2.0)$);
\draw (Q1)--(Q4)--(Q3)--(Q2)--cycle;
\end{scope}

  \draw[dashed, draw=black, fill=yellow, fill opacity=0.5] (A)--(B)--(C)--(D)--cycle;
\end{tikzpicture}
\caption{Chain of flats form alternating product region}
\label{Fig:Chain-of-Flats}
\end{figure}
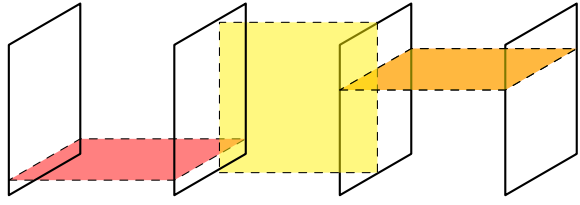

We now construct a path from $v$ to $w$ block by block.
Suppose we are at a point of the form
\[
p_{j-1}D_{\beta_j}^{\varepsilon_j N}.
\]
Read the block $h_j$ letter by letter. Since every letter in $h_j$ commutes with
$D_{\beta_j}$, this produces a path of length equal to the word-length of $h_j$ from
\[
p_{j-1}D_{\beta_j}^{\varepsilon_j N}
\quad\text{to}\quad
p_jD_{\beta_j}^{\varepsilon_j N}.
\]

For $1\le j\le r-1$, since $\beta_j$ and $\beta_{j+1}$ are disjoint, we connect
\[
p_jD_{\beta_j}^{\varepsilon_j N}
\quad\text{to}\quad
p_jD_{\beta_{j+1}}^{\varepsilon_{j+1}N}
\]
by first adding the new twist,
\[
p_jD_{\beta_j}^{\varepsilon_j N}
\longrightarrow
p_jD_{\beta_j}^{\varepsilon_j N}D_{\beta_{j+1}}^{\varepsilon_{j+1}m}
\qquad (0\le m\le N),
\]
and then removing the old twist,
\[
p_jD_{\beta_j}^{\varepsilon_j m}D_{\beta_{j+1}}^{\varepsilon_{j+1}N}
\longrightarrow
p_jD_{\beta_{j+1}}^{\varepsilon_{j+1}N}
\qquad (N\ge m\ge 0).
\]
This contributes length $2N$.

Concatenating all block-moves and all switch-moves yields a path from $v$ to $w$.
Its length is bounded by
\[
n+2(r-1)N\le n+2nN.
\]
Including the initial and terminal twist-ray segments gives a total path length at most
\[
2N+n+2nN.
\]
Since $n\le 4C_0R$ and $N\asymp R$, this is $O(R^2)$.

It remains to verify that the whole path stays outside a fixed linear ball about the identity.
There are three kinds of vertices to check.

\medskip
\noindent
\emph{Type 1: vertices on the initial and terminal twist-ray segments.}
These are outside $B(\id,\eta R-T_1)$ by construction.

\medskip
\noindent
\emph{Type 2: vertices encountered while reading a block $h_j$.}
Every such vertex has the form
\[
x=pD_{\beta_j}^{\varepsilon_j N},
\]
where $p$ is an intermediate prefix with $\Norm{p}_{\mathcal S}\le 6R$.
Hence Lemma~\ref{lem:extending-branch}\textup{(2)} gives
\[
\Norm{x}_{\mathcal S}
\ge \frac{N-6MR-A}{M}.
\]

\medskip
\noindent
\emph{Type 3: vertices encountered during a switch move.}
Every such vertex has the form
\[
x=p_jD_{\beta_j}^{\varepsilon_j m}D_{\beta_{j+1}}^{\varepsilon_{j+1}N}
\qquad\text{or}\qquad
x=p_jD_{\beta_j}^{\varepsilon_j N}D_{\beta_{j+1}}^{\varepsilon_{j+1}m}
\]
for some $0\le m\le N$.
Since $\beta_j$ and $\beta_{j+1}$ are disjoint and $\Norm{p_j}_{\mathcal S}\le 6R$,
Lemma~\ref{lem:extending-branch}\textup{(3)} yields
\[
\Norm{x}_{\mathcal S}
\ge \frac{N-6MR-A}{M}.
\]

Choose $K_{\mathrm{tw}}$ so large that
\[
\frac{K_{\mathrm{tw}}-6M}{M}\ge 2\eta.
\]
Then, for all sufficiently large $R$,
\[
\frac{N-6MR-A}{M}\ge \eta R.
\]
Choosing $T_0 \geq T_1$ large enough, we may assume that
\[
\frac{N-6MR-A}{M}\ge \eta R-T_0
\qquad\text{for all }R>0.
\]
Therefore every vertex on the detour from $v$ to $w$ lies outside
\[
B(\id,\eta R-T_0).
\]

Combining the initial twist-ray segment from $g_1$ to $v$, the detour from $v$ to $w$,
and the reverse of the twist-ray segment from $g_2$ to $w$, we obtain a path from
$g_1$ to $g_2$ which is disjoint from
\[
B(\id,\eta R-T_0)
\]
and has length at most $M_0R^2$ for some uniform constant $M_0$.
This proves the required quadratic upper bound.
\end{proof}
\subsection{End space of big mapping class groups}

It is a classical question to study the end space of a group. When the end space is nonempty,
the trichotomy of having $1$, $2$, or infinitely many ends is given by the
Freudenthal--Hopf theorem. Analogous trichotomy results have also been established for
countably generated and compactly generated groups in \cite{yves}, and for all non-locally finite
graphs in \cite{OP}.

In the setting of CB-generated big mapping class groups, the end space may be defined using the
Cayley graph associated to any CB word metric, by taking the inverse limit of the sets of connected
components of the complements of bounded balls. This question remained largely open until recently.
In \cite{OQW}, it was shown that for many avenue surfaces (defined in \cite{HQR}), the associated
mapping class group is one-ended. For these surfaces we have $\zeta(\Sigma)=2$.

We obtain a complementary result for a class of surfaces disjoint from those considered in
\cite{OQW}. For surfaces satisfying the assumptions of Theorem~\ref{introthmDiv}, one-endedness
follows from the quadratic divergence bound.

\begin{corollary}
Suppose $\Sigma$ is a stable surface and $\MapSig$ is CB-generated. Suppose additionally that
$\zeta(\Sigma)\ge 5$. Then, with respect to any CB-generating set, $\MapSig$ is one-ended.
\end{corollary}

\end{document}